\documentclass[12pt,a4paper]{article}
\usepackage{amsmath}
\usepackage{amssymb}
\usepackage{amssymb}
\usepackage{cases}
\usepackage{bbm}
\usepackage{mathrsfs}
\usepackage{graphicx}
\usepackage{amsfonts}
\usepackage{enumerate}
\usepackage{theorem}
\usepackage[pdfstartview=FitH]{hyperref}
\makeatletter
\def\tank#1{\protected@xdef\@thanks{\@thanks
        \protect\footnotetext[0]{#1}}}
\def\bigfoot{

    \@footnotetext}
\makeatother

\newcommand{\ea}{\end{array}}

\allowdisplaybreaks

\newtheorem{theorem}{Theorem}[section]
\newtheorem{lem}{Lemma}[section]
\newtheorem{prp}[theorem]{Proposition}
\newtheorem{thm}[theorem]{Theorem}

\newtheorem{dfn}[theorem]{Definition}
\newtheorem{exmp}[theorem]{Example}
\newtheorem{remark}{Remark}

\newtheorem{con}{Condition}[section]

{\theorembodyfont{\rmfamily}
}

\title{Large Deviations for SPDEs of Jump Type}

\author{ Xue Yang \quad Jianliang Zhai \quad  Tusheng Zhang \\ School of Mathematics, University of Manchester, \\Oxfor Road, Manchester, M13 9PL, UK}

\date{}
\newenvironment{proof}{\par\noindent{\bf Proof:}}{\hspace*{\fill}$\blacksquare$\par}
\begin{document}
\maketitle

\begin{minipage}{140mm}
\begin{center}
{\bf Abstract}
\end{center}
In this paper, we establish a large deviation principle for a fully non-linear stochastic evolution equation driven by both Brownian motions and Poisson random measures on a given Hilbert space $H$. The weak convergence
method plays an important role.
\end{minipage}

\vspace{4mm}


\vspace{3mm}
\section{Introduction}

In this paper, we are concerned with large deviation principles for stochastic evolution equations (stochastic partial differential equations (SPDEs) in particular) of jump type on some Hilbert space $H$:
\begin{eqnarray}\label{SPDE-01}
   X^\epsilon_t
       =
        X^\epsilon_0-\int^t_0\mathcal{A}(X^\epsilon_s)ds+\sqrt{\epsilon}\int_{0}^{t}\sigma(s,X^{\epsilon}_{s})d\beta(s)
        +\epsilon\int^t_0\int_\mathbb{X}G(s,X^\epsilon_{s-},v)\widetilde{N}^{\epsilon^{-1}}(dsdv).
\end{eqnarray}

Here $\mathcal{A}$ is an (normally unbounded) linear operator on $H$, $\mathbb{X}$ is a locally compact Polish space. $\beta=(\beta_{i})_{i=1}^{\infty}$ is an i.i.d. family of
standard Brownian motions.  $N^{\epsilon^{-1}}$ is a Poisson random measure on $[0,T]\times\mathbb{X}$ with a
$\sigma$-finite
mean measure $\epsilon^{-1}\lambda_T\otimes \nu$, $\lambda_T$ is the Lebesgue measure on $[0,T]$ and $\nu$ is a $\sigma-$finite measure on $\mathbb{X}$.
$\widetilde{N}^{\epsilon^{-1}}([0,t]\times B)=N^{\epsilon^{-1}}([0,t]\times B)-\epsilon^{-1}t\nu(B)$,
$\forall B\in \mathcal{B}(\mathbb{X})$ with $\nu(B)<\infty$, is the compensated Poisson random measure.
\vskip 0.3cm
Large deviations for stochastic evolution equations and stochastic
 partial differential equations driven by Gaussian processes have been
investigated in many papers, see e.g.\ \cite{CW}, \cite{CM2},\cite{CR}, \cite{S}, \cite{Z}. The situations for stochastic evolution equations and stochastic
 partial differential equations driven by L\'e{}vy noise are drastically different because of the appearance of the jumps. There is not much work on this topic so far. The first paper on large deviations of SPDEs of jump type is \cite{Rockner-Zhang} where the additive noise is considered. The case of multiplicative  L\'e{}vy noise is studied in \cite{Swiech-Zabczyk} where the large deviation was obtained on a larger space ( hence with a weaker topology ) than the actual state space of the solution. Recently, a new approach to large deviations of measurable maps of Poisson random measures (PRM) and Brownian motion (BM) was  introduced in \cite{Budhiraja-Dupuis-Maroulas.} based on variational representations of certain functionals of PRM and BM.
 One of the key elements in this approach
is to prove the weak convergence of  random perturbations of  the corresponding equations. So the underline topology is a very important factor to consider when establishing large deviations. In the new preprint \cite{Budhiraja-Chen-Dupuis}, the authors applied the criteria in \cite{Budhiraja-Dupuis-Maroulas.} to obtain  a large deviation principle
for stochastic partial differential equations driven by Poisson random measures on some nuclear spaces where tightness of measures are relatively easy to establish. Often the real state space of the solution of a stochastic partial differential equation is a smaller Hilbert space contained in the nuclear space. This makes it interesting to directly consider large deviations on  the actual state space.
\vskip 0.3cm
 The aim of this paper is to establish a large deviation principle for a fully non-linear stochastic evolution equation
driven by both Brownian motions and Poisson random measures like (\ref{SPDE-01}) on a given Hilbert space $H$. We will apply the  criteria in \cite{Budhiraja-Dupuis-Maroulas.}. Among other things , we need to  prove the tightness of the solutions of random perturbations of the equation (\ref{SPDE-01}) on the space $D([0,T]; H)$. To this end, we  split the time interval $[0,T]$ into $[0, t_0]$ and $[t_0, T]$ for a given arbitrarily small positive constant $t_0>0$ because two different treatments are needed for these two intervals. This also make the proofs involved.
\vskip 0.3cm
Finally we mention that  large
deviations for  L\'e{}vy processes on Banach spaces and large
deviations for solutions of stochastic differential equations driven
by Poisson measures in finite dimensions were studied in \cite{A1}, \cite{A2}.
\vskip 0.3cm
The rest of the paper is organized as follows. In Section 2, we recall the general criteria of large deviations obtained in \cite{Budhiraja-Dupuis-Maroulas.} and formulate precisely the stochastic evolution equations we are going to study. Section 3 is devoted to the proof of the large deviation principle. A number of preparing propositions and lemmas will be  proved in this section.

\vskip 0.3cm

We end this section with some notations. For a topological space $\mathcal{E}$, denote the corresponding Borel
$\sigma$-field by $\mathcal{B(\mathcal{E})}$. We will use the symbol $"\Longrightarrow"$ to denote convergence
in distribution. Let $\mathbb{N},\ \mathbb{N}_0,\ \mathbb{R},\ \mathbb{R}_+,\ \mathbb{R}^d$ denote the set of
positive integers, non-negative integers, real numbers, positive real numbers, and d-dimensional real
vectors respectively. For a Polish space $\mathbb{X}$, denote by $C([0,T],\mathbb{X})$, $D([0,T],\mathbb{X})$
the space of continuous functions and right continuous functions with left limits from [0,T] to $\mathbb{X}$
respectively.
For a metric space $\mathcal{E}$, denote by $M_b(\mathcal{E})$, $C_b(\mathcal{E})$ the space of real valued bounded
$\mathcal{B}(\mathcal{E})/\mathcal{B}(\mathbb{R})$-measurable maps and real valued  bounded continuous functions
respectively.
For $p>0$, a measure $\nu$ on $\mathcal{E}$, and a Hilbert space $H$, denote by $L^p(\mathcal{E},\nu;H)$ the space of
measurable
functions $f$ from $\mathcal{E}$ to $H$ such that $\int_{\mathcal{E}}\|f(v)\|^p\nu(dv)<\infty$, where $\|\cdot\|_{H}$ is
the
norm on $H$. For a function $x:[0,T]\rightarrow\mathcal{E}$, we use the notation $x_t$ and $x(t)$ interchangeably for
the evaluation of $x$ at $t\in[0,T]$. Similar convention will be followed for stochastic processes. We say a
collection $\{X^\epsilon\}$
of $\mathcal{E}$-valued random variables is tight if the probability distributions of $X^\epsilon$ are tight in
$\mathcal{P}(\mathcal{E})$
(the space of probability measures on $\mathcal{E}$).

\section{Preliminaries}
In the first part of this section, we will recall the general criteria for a large deviation principle given in \cite{Budhiraja-Dupuis-Maroulas.}. To this send,  we closely follow the framework and the notations  in \cite{Budhiraja-Chen-Dupuis} and \cite{Budhiraja-Dupuis-Maroulas.}. In the second part, we will precisely formulate
the stochastic evolution equations we will study.

\subsection{Large Deviation Principle}

Let $\{X^\epsilon,\epsilon>0\}\equiv\{X^\epsilon\}$ be a family of random variables defined on a probability space
$(\Omega,\mathcal{F},\mathbb{P})$
and taking values in a Polish space (i.e., a complete separable metric space) $\mathcal{E}$. Denote expectation with
respect to $\mathbb{P}$ by $\mathbb{E}$.
The theory of large deviations is concerned with events $A$ for which probability $\mathbb{P}(X^\epsilon\in A)$
converge to
zero exponentially fast as $\epsilon\rightarrow 0$. The exponential decay rate of such probabilities is typically
expressed
in terms of a ``rate function" $I$ mapping $\mathcal{E}$ into $[0,\infty]$.
    \begin{dfn}\label{Dfn-Rate function}
       \emph{\textbf{(Rate function)}} A function $I: \mathcal{E}\rightarrow[0,\infty]$ is called a rate function on
       $\mathcal{E}$,
       if for each $M<\infty$ the level set $\{x\in\mathcal{E}:I(x)\leq M\}$ is a compact subset of $\mathcal{E}$.
       For
       $A\in \mathcal{B}(\mathcal{E})$,
       we define $I(A)\doteq\inf_{x\in A}I(x)$.
    \end{dfn}
    \begin{dfn}
       \emph{\textbf{(Large deviation principle)}} Let $I$ be a rate function on $\mathcal{E}$. The sequence
       $\{X^\epsilon\}$
       is said to satisfy the large deviation principle on $\mathcal{E}$ with rate function $I$ if the following two
       conditions
       hold.

         a. Large deviation upper bound. For each closed subset $F$ of $\mathcal{E}$,
              $$
                \limsup_{\epsilon\rightarrow 0}\epsilon\log\mathbb{P}(X^\epsilon\in F)\leq-I(F).
              $$

         b. Large deviation lower bound. For each open subset $G$ of $\mathcal{E}$,
              $$
                \limsup_{\epsilon\rightarrow 0}\epsilon\log\mathbb{P}(X^\epsilon\in G)\geq-I(G).
              $$
    \end{dfn}

If a sequence of random variables satisfies a large deviation principle with some rate function,
then the rate function is unique.

\subsection{Poisson Random Measure and Brownian Motion}\label{Section Representation}

\subsubsection{Poisson Random Measure}
\label{Poisson Random Measure}
Let $\mathbb{X}$ be a locally compact Polish space. Let $\mathcal{M}_{FC}(\mathbb{X})$ be the space of all
measures $\nu$ on $(\mathbb{X},\mathcal{B}(\mathbb{X}))$ such that $\nu(K)<\infty$ for every compact $K$ in
$\mathbb{X}$. Endow $\mathcal{M}_{FC}(\mathbb{X})$ with the weakest topology such that for every $f\in
C_c(\mathbb{X})$
(the space of continuous functions with compact supports), the function
$\nu\rightarrow\langle f,\nu\rangle=\int_{\mathbb{X}}f(u)d\nu(u),\nu\in\mathcal{M}_{FC}(\mathbb{X})$ is continuous.
This topology can be metrized such that $\mathcal{M}_{FC}(\mathbb{X})$ is a Polish space (see e.g.
\cite{Budhiraja-Dupuis-Maroulas.}).
Fix $T\in(0,\infty)$ and let $\mathbb{X}_T=[0,T]\times\mathbb{X}$. Fix a measure
$\nu\in\mathcal{M}_{FC}(\mathbb{X})$,
and let $\nu_T=\lambda_T\otimes\nu$, where $\lambda_T$ is Lebesgue measure on $[0,T]$.

We recall that a Poisson random measure $\textbf{n}$ on $\mathbb{X}_T$ with mean measure (or intensity measure)
$\nu_T$ is a $\mathcal{M}_{FC}(\mathbb{X}_T)$ valued random variable such that for each
$B\in\mathcal{B}(\mathbb{X}_T)$
with $\nu_T(B)<\infty$, $\textbf{n}(B)$ is Poisson distributed with mean $\nu_T(B)$ and for disjoint
$B_1,\cdots,B_k\in\mathcal{B}(\mathbb{X}_T)$, $\textbf{n}(B_1),\cdots,\textbf{n}(B_k)$ are mutually independent
random
variables (cf. \cite{Ikeda-Watanabe}). Denote by $\mathbb{P}$ the measure induced by $\textbf{n}$ on
$(\mathcal{M}_{FC}(\mathbb{X}_T),\mathcal{B}(\mathcal{M}_{FC}(\mathbb{X}_T)))$.
Then letting $\mathbb{M}=\mathcal{M}_{FC}(\mathbb{X}_T)$, $\mathbb{P}$ is the unique probability measure on
$(\mathbb{M},\mathcal{B}(\mathbb{M}))$
under which the canonical map, $N:\mathbb{M}\rightarrow\mathbb{M},\ N(m)\doteq m$, is a Poisson random measure with
intensity measure $\nu_T$. With applications to large deviations in mind, we also consider, for $\theta>0$,
probability
measures $\mathbb{P}_\theta$ on $(\mathbb{M},\mathcal{B}(\mathbb{M}))$ under which $N$ is a Poissson random measure
with intensity $\theta\nu_T$. The corresponding expectation operators will be denoted by $\mathbb{E}$ and
$\mathbb{E}_\theta$,
respectively.



Let $\mathbb{Y}=\mathbb{X}\times[0,\infty)$ and $\mathbb{Y}_T=[0,T]\times\mathbb{Y}$. Let
$\bar{\mathbb{M}}=\mathcal{M}_{FC}(\mathbb{Y}_T)$
and let $\bar{\mathbb{P}}$ be the unique probability measure on $(\bar{\mathbb{M}},\mathcal{B}(\bar{\mathbb{M}}))$
under which the canonical map, $\bar{N}:\bar{\mathbb{M}}\rightarrow\bar{\mathbb{M}},\bar{N}(m)\doteq m$, is a Poisson
random
measure with intensity measure $\bar{\nu}_T=\lambda_T\otimes\nu\otimes \lambda_\infty$, with $\lambda_\infty$ being
Lebesgue measure on $[0,\infty)$.
The corresponding expectation operator will be denoted by $\bar{\mathbb{E}}$. Let
$\mathcal{F}_t\doteq\sigma\{\bar{N}((0,s]\times A):0\leq s\leq t,A\in\mathcal{B}(\mathbb{Y})\},$ and let
$\bar{\mathcal{F}}_t$
denote the completion under $\bar{\mathbb{P}}$. We denote by $\bar{\mathcal{P}}$ the predictable $\sigma$-field on
$[0,T]\times\bar{\mathbb{M}}$
with the filtration $\{\bar{\mathcal{F}}_t:0\leq t\leq T\}$ on $(\bar{\mathbb{M}},\mathcal{B}(\bar{\mathbb{M}}))$.
Let
$\bar{\mathcal{A}}$
be the class of all $(\bar{\mathcal{P}}\otimes\mathcal{B}(\mathbb{X}))/\mathcal{B}[0,\infty)$-measurable maps
$\varphi:\mathbb{X}_T\times\bar{\mathbb{M}}\rightarrow[0,\infty)$. For $\varphi\in\bar{\mathcal{A}}$, define a
counting process $N^\varphi$ on
$\mathbb{X}_T$ by
   \begin{eqnarray}\label{Jump-representation}
      N^\varphi((0,t]\times U)=\int_{(0,t]\times U}\int_{(0,\infty)}1_{[0,\varphi(s,x)]}(r)\bar{N}(dsdxdr),\
      t\in[0,T],U\in\mathcal{B}(\mathbb{X}).
   \end{eqnarray}

$N^\varphi$ is  the controlled random measure, with $\varphi$ selecting the intensity for the points at location
$x$
and time $s$, in a possibly random but non-anticipating way. When $\varphi(s,x,\bar{m})\equiv\theta\in(0,\infty)$, we
write $N^\varphi=N^\theta$. Note that $N^\theta$ has the same distribution with respect to $\bar{\mathbb{P}}$ as $N$
has with respect to $\mathbb{P}_\theta$.

\subsubsection{PRM and BM}
\label{section-PRM AND BM}
Denote $C([0,T],\mathbb{R}^{\infty})$ by $\mathbb{W}$, where $\mathbb{R}^{\infty}$ is the infinite
product space of the real line $\mathbb{R}$ and endowed with the product topology. Let
$\mathbb{V}=\mathbb{W}\times \mathbb{M}$ . Then let the  mapping $N: \mathbb{V}\rightarrow \mathbb{M}$  be defined by
$N(w,m)=m$ for $(w,m)\in  \mathbb{V}$, and let $\beta=(\beta_{i})_{i=1}^{\infty}$ be defined by $\beta_{i}(w,m)=w_{i}$
for
$(w,m)\in  \mathbb{V}$. Define the $\sigma-$ filtration $\mathcal{G}_{t}:=\sigma\{N((0,s]\times A),\beta_{i}(s):
0\leq
s\leq t, A\in \mathcal{B}(\mathbb{Y}), i\geq 1\}$. For every $\theta>0$, $\mathbb{P}_{\theta}$ denotes the unique
probability measure on $(\mathbb{V},\mathcal{B}(\mathbb{V}))$ such that :\\
(a) $(\beta_{i})_{i=1}^{\infty}$  is an i.i.d. family of standard Brownian motions.\\
(b) $N$ is a PRM with intensity measure $\theta\nu_{T}$.\\

If controlled Poisson random measure is also considered, we set $\bar{\mathbb{V}}:=\mathbb{W}\times
\bar{\mathbb{M}}$,
and let the mapping  $\bar{N}: \bar{\mathbb{V}}\rightarrow \bar{\mathbb{M}}$ be defined by
$\bar{N}(\bar{w},\bar{m})=\bar{m}$ for $(\bar{w},\bar{m})\in  \bar{\mathbb{V}}$ accordingly. Analogously, we define
$(\bar{\mathbb{P}}_{\theta},\bar{\mathcal{G}}_{t})$. We denote by $\{\bar{\mathcal{F}}_{t}\}$ the
$\bar{\mathbb{P}}-$completion of $\{\bar{\mathcal{G}}_{t}\}$ and $\bar{\mathcal{P}}$ the predictable $\sigma-$filed
on
$[0,T]\times \bar{\mathbb{V}}$ with the filtration $\{\bar{\mathcal{F}}_{t}\}$ on
$(\bar{\mathbb{V}},B(\bar{\mathbb{V}}))$.  Let $\bar{\mathcal{A}}$
be the class of all $(\bar{\mathcal{P}}\otimes\mathcal{B}(\mathbb{X}))/\mathcal{B}[0,\infty)$-measurable maps
$\varphi:\mathbb{X}_T\times\bar{\mathbb{V}}\rightarrow[0,\infty)$.
Define $l:[0,\infty)\rightarrow[0,\infty)$ by
    $$
    l(r)=r\log r-r+1,\ \ r\in[0,\infty).
    $$

For any $\varphi\in\bar{\mathcal{A}}$ the quantity
    \begin{eqnarray}\label{L_T}
      L_T(\varphi)=\int_{\mathbb{X}_T}l(\varphi(t,x,\omega))\nu_T(dtdx)
    \end{eqnarray}
is well defined as a $[0,\infty]$-valued random variable.

Let $H$ be a separable Hilbert space.

Define function space
\begin{eqnarray}
\mathcal{L}_{2}:=\{\psi: \psi\; \text{is}\; \bar{\mathcal{P}}\setminus
\mathcal{B}(H)\; \text{measurable and} \; \int_{0}^{T}\|\psi(s)\|_{H}^{2}\;ds<\infty,a.s.-
\bar{\mathbb{P}}\}.
\end{eqnarray}

Set $\mathcal{U}=\mathcal{L}_{2}\times \bar{\mathcal{A}}$. Define
$\tilde{L}_{T}(\psi):=\frac{1}{2}\int_{0}^{T}\|\psi(s)\|_{H}^{2}\;ds$ for $\psi\in \mathcal{L}_{2}$,
and $\bar{L}_{T}(u):=\tilde{L}_{T}(\psi)+L_{T}(\varphi)$ for $u=(\psi,\varphi)$.

\subsection{A General Criteria}

In this section, we recall a general criteria for a large deviation principle established in \cite{Budhiraja-Dupuis-Maroulas.}. Let
$\{\mathcal{G}^\epsilon\}_{\epsilon>0}$
be a family of measurable maps from $\bar{\mathbb{V}}$ to $\mathbb{U}$, where $\bar{\mathbb{V}}$ is introduced in Section $\ref{Poisson Random Measure}$ and $\mathbb{U}$ is some Polish space. We present
below a sufficient condition for large deviation principle (LDP in abbreviation) to hold for the family
$Z^\epsilon=\mathcal{G}^\epsilon(\sqrt{\epsilon}\beta, \epsilon N^{\epsilon^{-1}})$,
as $\epsilon\rightarrow 0$.

Define
   \begin{eqnarray}\label{S_N}
     S^N=\{g:\mathbb{X}_T\rightarrow[0,\infty):\,L_T(g)\leq N\},
   \end{eqnarray}
and
\begin{eqnarray}
 \tilde{S}^N=\{f:L^{2}([0,T]:H):\,\tilde{L}_T(f)\leq N\}.
\end{eqnarray}

A function $g\in S^N$ can be identified with a measure $\nu_T^g\in\mathbb{M}$, defined by
   \begin{eqnarray*}
      \nu_T^g(A)=\int_A g(s,x)\nu_T(dsdx),\ \ A\in\mathcal{B}(\mathbb{X}_T).
   \end{eqnarray*}

This identification induces a topology on $S^N$ under which $S^N$ is a compact space, see the Appendix of
\cite{Budhiraja-Chen-Dupuis}.
Throughout we use this topology on $S^N$. Set $\bar{S}^{N}=\tilde{S}^{N}\times S^{N}$. Define
$\mathbb{S}=\bigcup_{N\geq 1}\bar{S}^N$, and let
    $$
      \mathcal{U}^N=\{u=(\psi,\varphi)\in \mathcal{U}: u(\omega)\in \bar{S}^N,\bar{\mathbb{P}}\ a.e.\ \omega\},
    $$
where $\mathcal{U}$ is introduced in Section $\ref{section-PRM AND BM}$.

The following condition will be sufficient to establish a LDP for a family $\{Z^\epsilon\}_{\epsilon>0}$ defined
by $Z^\epsilon=\mathcal{G}^\epsilon(\sqrt{\epsilon}\beta, \epsilon N^{\epsilon^{-1}})$.

\begin{con}\label{LDP}
 There exists a measurable map $\mathcal{G}^0:\mathbb{V}\rightarrow \mathbb{U}$ such that the following hold.

 a. For $N\in\mathbb{N}$, let $(f_{n},g_n),\ (f,g)\in \bar{S}^N$ be such that $(f_{n},g_n)\rightarrow (f,g)$ as
 $n\rightarrow\infty$. Then
      $$
         \mathcal{G}^0(\int_{0}^{\cdot}f_{n}(s)ds, \nu_T^{g_n})\rightarrow \mathcal{G}^0(\int_{0}^{\cdot}f(s)ds,
         \nu_T^{g})\quad\text{in}\quad \mathbb{U}.
      $$

 b. For $N\in\mathbb{N}$, let $u_{\epsilon}=(\psi_{\epsilon},\varphi_\epsilon),\ u=(\psi,\varphi)\in\mathcal{U}^N$ be
 such that $u_\epsilon$
 converges in distribution to $u$ as $\epsilon\rightarrow 0$. Then
      $$
         \mathcal{G}^\epsilon(\sqrt{\epsilon}\beta+\int_{0}^{\cdot}\psi_{\epsilon}(s)ds, \epsilon
         N^{\epsilon^{-1}\varphi_\epsilon})\Rightarrow
         \mathcal{G}^0(\int_{0}^{\cdot}\psi(s)ds,\nu_T^{\varphi}).
      $$
\end{con}


For $\phi\in\mathbb{U}$, define $\mathbb{S}_\phi=\{(f,g)\in\mathbb{S}:\phi=\mathcal{G}^{0}(\int_{0}^{\cdot}f(s)ds,
\nu^g_T)\}$. Let $I:\mathbb{U}\rightarrow[0,\infty]$
be defined by
     \begin{eqnarray}\label{Rate function I}
        I(\phi)=\inf_{q=(f,g)\in\mathbb{S}_\phi}\{\bar{L}_T(q)\},\ \ \phi\in\mathbb{U}.
     \end{eqnarray}
By convention, $I(\phi)=\infty$ if $\mathbb{S}_\phi=\emptyset$.

The following criteria was established in \cite{Budhiraja-Dupuis-Maroulas.}.

\begin{thm}\label{LDP-main}
For $\epsilon>0$, let $Z^\epsilon$ be defined by $Z^\epsilon=\mathcal{G}^\epsilon(\sqrt{\epsilon}\beta,\epsilon
N^{\epsilon^{-1}})$, and suppose
that Condition \ref{LDP} holds. Then $I$ defined as in (\ref{Rate function I}) is a rate function on $\mathbb{U}$ and
the family $\{Z^\epsilon\}_{\epsilon>0}$ satisfies a large deviation principle with rate function $I$.
\end{thm}

For applications, the following strengthened form of Theorem \ref{LDP-main} is useful. Let $\{K_n\subset \mathbb{X},\
n=1,2,\cdots\}$
be an increasing sequence of compact sets such that $\cup _{n=1}^\infty K_n=\mathbb{X}$. For each $n$ let
   \begin{eqnarray*}
     \bar{\mathcal{A}}_{b,n}
          &\doteq&
              \{\varphi\in\bar{\mathcal{A}}:
                                   \ for\ all\ (t,\omega)\in[0,T]\times\bar{\mathbb{M}},\
                                   n\geq\varphi(t,x,\omega)\geq
                                   1/n\ if\ x\in K_n\ \\
                                  &&\ \ \ \ \ \ \ \ \ \ \ \ \ \ \ \ \ \ \ \ \ \ \ \ \ \ \ \ \ \ \ \ \ \ \ \ \ \ \ \ \
                                  \ \ \ \
                                       and\ \varphi(t,x,\omega)=1\ if\ x\in K_n^c
              \},
   \end{eqnarray*}
and let $\bar{\mathcal{A}}_{b}=\cup _{n=1}^\infty\bar{\mathcal{A}}_{b,n}$. Define
$\tilde{\mathcal{U}}^N=\mathcal{U}^N\cap\{(\psi,\phi): \phi\in\bar{\mathcal{A}}_b\}$.

\begin{thm}\label{LDP-main-01}
  Suppose Condition \ref{LDP} holds with $\mathcal{U}^N$ replaced by $\tilde{\mathcal{U}}^N$. Then the conclusions of
  Theorem \ref{LDP-main}
  continue to hold .
\end{thm}

\subsection{SPDEs}
\label{section SPDEs}
In this section we introduce the stochastic partial differential equations (SPDEs in addreviation) that will be studied
in this paper. Let $H$, $V$ be two separable Hilbert spaces such that $V$ is continuously, densely imbedded in $H$. Identifying
$H$ with its dual we have
    $$
      V\subset H\cong H'\subset V',
    $$
where $V'$ stands for the topological dual of $V$. Let $\mathcal{A}$ be a bounded linear operator from $V$ to $V'$ satisfying
the following coercivity hypothesis: There exist constants $\alpha>0$ and $\lambda_0\geq0$ such that
\begin{eqnarray}
\label{condition-lower bounds of A}
   2\langle \mathcal{A}u,u\rangle+\lambda_0\|u\|^2_H\geq\alpha\|u\|^2_V,\ for\ all\ u\in V.
\end{eqnarray}

\begin{exmp}
Let $H=L^2(D)$, where $D\subset \mathbb{R}^d$ is a bounded domain,  and set
$$
V=H^{1,2}_{0}(D)=\overline{C^{\infty}_{0}(D)}^{\|\cdot \|},
$$
where $C^{\infty}_{0}(D)$ is the space of infinite differentiable functions with compact supports and the norm is defined as follows
$$
\|f\|^{2}:=\|f\|^{2}_{L^2}+\|\nabla f\|^{2}_{L^2}.
$$

Denote by $a(x)=(a_{ij}(x))$  a matrix-valued function on $D$
satisfying  the uniform ellipticity condition:
$$
\frac{1}{c} I_d \leq a(x)\leq c I_d\qquad \hbox{for some constant
$\;c\in(0,\infty)$}.
$$
Let $b(x)$ be a vector field on $D$ with $b\in L^p(D)$ for
some $p>d$. Define
$$
\mathcal{A}u=-div(a(x)\nabla u(x))+b(x)\cdot \nabla u(x).
$$
Then (\ref{condition-lower bounds of A}) is fulfilled for $(H,V,\mathcal{A})$.
\end{exmp}

\begin{exmp}
\rm \ Stochastic evolution equations associated with fractional
Laplacian:
\begin{eqnarray}
dY_t&=&\Delta_{\alpha}Y_t dt+dL_t, \\
Y_0&=&h\in H,
\end{eqnarray}
where  $\Delta_{\alpha}$ denotes  the generator of the symmetric
$\alpha$-stable process in $R^d$, $0<\alpha \leq 2$.
$\Delta_{\alpha}$ is called the fractional Laplace operator. $L_t$
stands for a L\'e{}vy process. It is well known that the Dirichlet
form associated with $\Delta_{\alpha}$ is given by
$${\cal E}(u,v)=K(d,\alpha )\int \int_{R^d\times R^d} \frac{(u(x)-u(y))(v(x)-v(y))}{|x-y|^{d+\alpha}}\,dxdy,$$
$$D({\cal E})=\{u\in L^2(R^d):\quad \int \int_{R^d\times R^d} \frac{|u(x)-u(y)|^2}{|x-y|^{d+\alpha}}\,dxdy<\infty \},$$
where $K(d,\alpha )=\alpha 2^{\alpha
-3}\pi^{-\frac{d+2}{2}}sin(\frac{\alpha\pi}{2})\Gamma
(\frac{d+\alpha}{2})\Gamma (\frac{\alpha}{2})$. We choose  $H=L^2(\mathbb{R}^d)$, and $V=D({\cal E})$ with the inner
product $<u,v>={\cal E}(u,v)+(u,v)_{L^2(R^d)}$.

\noindent Define
$$
\mathcal{A}u=-\Delta_{\alpha}.
$$
Then (\ref{condition-lower bounds of A}) is fulfilled for $(H,V,\mathcal{A})$. See \cite{FOT} for details about
the fractional Laplace operator.
\end{exmp}
\vspace{3mm}

Assume that $\mathcal{A}^{*}$ the adjoint operator of $\mathcal{A}$, admits a complete system of eigenvectors; that is, there exists a sequence $\{e_{k},k\geq 1\}\subset V$ that forms an orthonormal basis of $H$ such that
$$
\mathcal{A}^{*}e_{k}=\zeta_{k}e_{k}\quad \text{for} k\geq1.
$$

We assume
$
0\leq\zeta_1\leq\zeta_2\leq\cdots\rightarrow\infty
$
and denote by $H_{\mathcal{A}^{*}}=\{h\in H: \|\mathcal{A}^{*}h\|_{H}^{2}<\infty\}$ the domain of $\mathcal{A}^{*}$. Suppose that the $H$
cylindrical Brownian motion $\beta$
admits the following representation:
$$
\beta_t=\sum_{k=1}^\infty\beta_k(t)e_k
$$
where $\beta_k(t),k\geq1$ are independent standard Brownian motions.
%

Denote by $L_2(H)$ the space of all Hilbert-Schmidt operators from $H$ to $H$.
Let $\sigma:[0,T]\times H\rightarrow L_2(H)$, $G:[0,T]\times H\times\mathbb{X}\rightarrow H$ be maps satisfying the following conditions:
\begin{con}
\label{Condition G}
\mbox{}\par
There exists $K(\cdot)\in L^{1}([0,T],\mathbb{R}^+)$ such that

(1)(Growth) For all $t\in [0,T]$, and $u\in H$,
    $$
      \|\sigma(t,u)\|^2_{L_2(H)}+\int_{\mathbb{X}}\|G(t,u,v)\|^2_H\nu(dv)\leq K(t)(1+\|u\|^2_H);
    $$

$(2)$ (Lipschitz) For all $t\in[0,T]$, and $u_1,\ u_2\in H$,
    $$
         \|\sigma(t,u_1)-\sigma(t,u_2)\|^2_{L_2(H)}+\int_{\mathbb{X}}\|G(t,u_1,v)-G(t,u_2,v)\|^2_H\nu(dv)\leq K(t)\|u_1-u_2\|^2_H.
    $$

\end{con}

Consider the following stochastic evolution equation:
\begin{eqnarray}
\label{SPDE}
X^\epsilon_t
       =
        X^\epsilon_0-\int^t_0\mathcal{A}X^\epsilon_s ds+\sqrt{\epsilon}\int_{0}^{t}\sigma(s,X^{\epsilon}_{s})d\beta(s)
        +\epsilon\int^t_0\int_\mathbb{X}G(s,X^\epsilon_{s-},v)\widetilde{N}^{\epsilon^{-1}}(dsdv).
\end{eqnarray}

Here the precise definition of the solution to  (\ref{SPDE}) is as follows.

\begin{dfn}\label{Def-Solution-001}
Let $(\bar{\mathbb{V}},\mathcal{B}(\bar{\mathbb{V}}),\bar{\mathbb{P}},\{\bar{\mathcal{F}_t}\})$ be the filtered probability
space
described in Section \ref{Section Representation}. Suppose that $X_0$ is a $\bar{\mathcal{F}_0}$-measurable $H$-valued
random variable
such that $\bar{\mathbb{E}}\|X_0\|^2_H<\infty$. A stochastic process $\{X^\epsilon_t\}_{t\in[0,T]}$ defined on $\bar{\mathbb{V}}$
is said
to be a $H$-valued solution to (\ref{SPDE}) with initial value $X_0$, if

$a)$ $X^\epsilon_t$ is a $H$-valued $\bar{\mathcal{F}}_t$-measurable random variable, for all $t\in[0,T]$;

$b)$ $X^\epsilon\in D([0,T],H)\cap L^2([0,T],V)$ a.s.;

$c)$ For all $t\in[0,T]$, every $\phi\in V$,
     \begin{eqnarray}\label{Def-Solution-001-weak}
       \langle X^\epsilon_t,\phi\rangle
           &=&
             \langle X_0,\phi\rangle-\int_0^t\langle \mathcal{A}X^\epsilon_s,\phi\rangle ds
             +\sqrt{\epsilon}\langle\int_{0}^{t}\sigma(s,X^{\epsilon}_{s})d\beta(s),\phi\rangle\nonumber\\
            && +\epsilon \int_0^t\int_{\mathbb{X}}\langle
            G(s,X_{s-}^\epsilon,v),\phi\rangle\widetilde{N}^{\epsilon^{-1}}(ds,dv), \text{a.s.}.
     \end{eqnarray}
    \end{dfn}

    \begin{dfn}\label{Def-Solution-001-Unique}$(\bf{Pathwise\ uniqueness})$
       We say that the $H$-valued solution for the stochastic evolution equation (\ref{SPDE}) has the pathwise
       uniqueness if any two
       $H$-valued solutions $X$ and $X'$ defined on the same filtered probability space with
       respect to the same Poisson random measure and Brownian motion starting from the same initial condition $X_0$ coincide almost
       surely.
    \end{dfn}

 \section{Large Deviation Principle}
 Assume $X_{0}$ is deterministic. Let $X^\epsilon$ be the $H$-valued solution to  (\ref{SPDE}) with initial value $X_0$. In this
 section, we establish an LDP for $\{X^\epsilon\}$ under suitable assumptions.

 We begin by introducing the map $\mathcal{G}_0$ that will be used to define the rate function and also used for
 verification of
 Condition \ref{LDP}. Recall that $\mathbb{S}=\bigcup_{N\geq1} \bar{S}^N$, where $\bar{S}^N$ is defined in last
 section. As a first step we show
 that under the conditions below, for every $q=(f,g)\in \mathbb{S}$, the deterministic integral equation
     \begin{eqnarray}\label{SPDE-determined equation}
        \widetilde{X}^q_t
               =
                 X_0
               -
                 \int_0^t\mathcal{A}\widetilde{X}^q_sds
               +
               \int_{0}^{t}\sigma(s,  \widetilde{X}^q_s)f(s)ds
               +
                 \int_0^t\int_{\mathbb{X}}G(s,\widetilde{X}^q_s,v)(g(s,v)-1)\nu(dv)ds
     \end{eqnarray}
has a unique continuous solution. Here $q=(f,g)$ plays the role of a control.

Let
     $$
       \|G(t,v)\|_{0,H}=\sup_{u\in H}\frac{\|G(t,u,v)\|_H}{1+\|u\|_H},\ \ (t,v)\in[0,T]\times \mathbb{X}.
     $$
     $$
       \|G(t,v)\|_{1,H}=\sup_{u_1,u_2\in H,u_1\neq u_2}\frac{\|G(t,u_1,v)-G(t,u_2,v)\|_H}{\|u_1-u_2\|_H},\ \
       (t,v)\in[0,T]\times \mathbb{X}.
     $$

\begin{con}$(\bf{Exponential\ Integrability})$\label{Condition Exponential Integrability}
For $i=0,\ 1$, there exists $\delta^i_1>0$ such that for all $E\in\mathcal{B}([0,T]\times \mathbb{X})$ satisfying
$\nu_T(E)<\infty$, the following holds
         $$
           \int_Ee^{\delta^i_1 \|G(s,v)\|^2_{i,H}}\nu(dv)ds<\infty.
         $$
\end{con}

\begin{remark}\label{Remark-01}
Condition \ref{Condition Exponential Integrability} implies that, for every $\delta_2>0$ and for all
$E\in\mathcal{B}([0,T]\times \mathbb{X})$
satisfying $\nu_T(E)<\infty$,
   $$
     \int_Ee^{\delta_2\|G(s,v)\|_{0,H}}\nu(dv)ds<\infty.
   $$
\end{remark}

\vspace{4mm}
  Now recall the following inequalities from \cite{Budhiraja-Chen-Dupuis}, which will be used later.

  $a)$ For $a,b,\sigma\in(0,\infty)$, there exists $C(\sigma)$ only depending on $\sigma$, such that
     \begin{eqnarray}\label{Inequality-001}
       ab\leq C(\sigma)e^{\sigma a}+\frac{1}{\sigma}(b\log b-b+1)=C(\sigma)e^{\sigma a}+\frac{1}{\sigma}l(b);
     \end{eqnarray}

  $b)$ For each $\beta>0$ there exists $c_1(\beta)>0$, such that $c_1(\beta)\rightarrow 0$ as $\beta\rightarrow\infty$
  and
      $$
          |x-1|\leq c_1(\beta)l(x)\ whenever\ |x-1|\geq\beta;
      $$

  $c)$ For each $\beta>0$ there exists $c_2(\beta)<\infty$, such that
        $$
          |x-1|^2\leq c_2(\beta)l(x)\ whenever\ |x-1|\leq\beta.
        $$

The following lemma was proved in  \cite{Budhiraja-Chen-Dupuis}.
\begin{lem}\label{Lemma-Condition-0,H-1,H}
Under Condition \ref{Condition G} and Condition \ref{Condition Exponential Integrability}, for $i=0,\ 1$ and every
$N\in\mathbb{N}$,
     \begin{eqnarray}\label{Inq-G-0-1}
        C^N_{i,2}:=\sup_{g\in S^N}\int_{\mathbb{X}_T}\|G(s,v)\|^2_{i,H}(g(s,v)+1)\nu(dv)ds<\infty,
     \end{eqnarray}

     \begin{eqnarray}\label{Inq-G-0-2}
        C^N_{i,1}:=\sup_{g\in S^N}\int_{\mathbb{X}_T}\|G(s,v)\|_{i,H}|g(s,v)-1|\nu(dv)ds<\infty.
     \end{eqnarray}

\end{lem}

We also need the following lemma whose proof can be found in  Chapter III of \cite{Temam}.

\begin{lem}\label{Lemma-thm-01}
Assume that         $$
                            \left\{
                                  \begin{array}{ll}
                                      \!\!\! f=f_1+f_2,\ f_1\in L^2([0,T],V'),\ f_2\in L^1([0,T],H),\\
                                      \!\!\! u_0\in H,
                                  \end{array}
                            \right.
                           $$
then there exists a unique function $u$ (denote by $u'$ its derivative) which satisfies
                          $$
                            \left\{
                                  \begin{array}{lll}
                                      \!\!\! u\in L^2([0,T],V)\cap C([0,T],H),u'\in L^2([0,T],V')+L^1([0,T],H),\\
                                      \!\!\! \langle u_t,\phi\rangle
                                             =
                                               \langle u_0,\phi\rangle
                                              -
                                               \int_0^t\langle \mathcal{A}u_s,\phi\rangle ds
                                              +
                                               \int_0^t\langle f_s,\phi\rangle ds,\ \ \forall\phi\in V,\\
                                      \!\!\! \frac{d}{dt}\langle u,u\rangle=2\langle u',u\rangle.
                                  \end{array}
                            \right.
                           $$
\end{lem}

\begin{lem}\label{Lemma-thm-02}
 a). If $Y\in C([0,T],H)$, for any $q=(f,g)\in\mathbb{S}$,  then
       $$
         \sigma(\cdot,Y(\cdot))f(\cdot)\in L^1([0,T],H),\ \int_{\mathbb{X}}G(\cdot,Y(\cdot),v)(g(\cdot,v)-1)\nu(dv)\in L^1([0,T],H);
       $$
 b). If $Y_n\in C([0,T],H)$, $n\geq1$ with $C=\sup_{n}\sup_{s\in[0,T]}\|Y_n(s)\|_H<\infty$, then
       \begin{eqnarray*}
          \widetilde{C}_N&:=&\sup_{q=(f,q)\in
          \bar{S}^N}\sup_n\Big[\int_0^T\|\int_{\mathbb{X}}G(s,Y_n(s),v)(g(s,v)-1)\nu(dv)\|_Hds+\int_0^T\|\sigma(s,Y_n(s))f(s)\|_Hds\Big]\\
          &<&\infty.
       \end{eqnarray*}
\end{lem}
\begin{proof}
Since
   \begin{eqnarray*}
          &  & \int_0^t\|\int_{\mathbb{X}}G(s,Y(s),v)(g(s,v)-1)\nu(dv)\|_Hds\\
          &\leq&
          \int_0^t\int_\mathbb{X}\|G(s,Y(s),v)(g(s,v)-1)\|_H\nu(dv)ds\nonumber\\
      &=&
          \int_0^t\int_\mathbb{X}\frac{\|G(s,Y(s),v)\|_H}{1+\|Y(s)\|_H}|g(s,v)-1|(1+\|Y(s)\|_H)\nu(dv)ds\nonumber\\
      &\leq&
          (1+\sup_{s\in[0,T]}\|Y(s)\|_H)\int_0^t\int_{\mathbb{X}}\|G(s,v)\|_{0,H}|g(s,v)-1|\nu(dv)ds,\nonumber
   \end{eqnarray*}
and by Condition \ref{Condition G},
   \begin{eqnarray*}
           \int_0^t\|\sigma(s,Y(s))f(s)\|_Hds
       &\leq&
           \int_0^t\|\sigma(s,Y(s))\|_{L_2(H)}\|f(s)\|_Hds\\
       &\leq&
           \int_0^t\|\sigma(s,Y(s))\|^2_{L_2(H)}ds+\int_0^t\|f(s)\|^2_Hds\\
       &\leq&
           \int_0^tK(s)\|Y(s)\|^2_{H}ds+\int_0^t\|f(s)\|^2_Hds\\
       &\leq&
           \Big(\sup_{s\in[0,T]}\|Y(s)\|^2_{H}\Big)\int_0^TK(s)ds+\int_0^T\|f(s)\|^2_Hds,
   \end{eqnarray*}
the lemma follows from Lemma \ref{Lemma-Condition-0,H-1,H}.
\end{proof}
\vspace{3mm}

\begin{thm}\label{Main-thm-01}
  Fix $q=(f,g)\in\mathbb{S}$. Suppose Condition \ref{Condition G} and Condition \ref{Condition Exponential Integrability}
  hold. Then there exists a unique $\widetilde{X}^q\in C([0,T],H)$ such that for every $\phi\in V$,
     \begin{eqnarray}\label{main-thm-01-eq}
       \langle \widetilde{X}^q_t,\phi\rangle
          =
            \langle X_0,\phi\rangle
          &  -&
            \int_0^t\langle \widetilde{X}^q_s,\mathcal{A}^{*}\phi\rangle ds
          +
          \int_0^t\langle\sigma(s,\widetilde{X}^q_s)f(s),\phi\rangle ds\nonumber\\
          &  +&
            \int_0^t\int_{\mathbb{X}}\langle G(s,\widetilde{X}^q_s,v),\phi\rangle(g(s,v)-1)\nu(dv)ds.
     \end{eqnarray}
Moreover, for fixed $N\in\mathbb{N}$, there exists $C_N>0$ such that
     \begin{eqnarray}\label{main-thm-01-ineq}
        \sup_{q\in S^N}\Big(\sup_{s\in[0,T]}\|\widetilde{X}^q_s\|^2_H+\int_0^T\|\widetilde{X}^q_s\|^2_Vds\Big)\leq
        C_N.
     \end{eqnarray}
\end{thm}
\begin{proof}
{\bf Existence\ of\ solution:}
Let $Y_0(t):=X_0,t\geq0$. Suppose $Y_{n-1}$ has been defined,
by Lemma \ref{Lemma-thm-01} and Lemma \ref{Lemma-thm-02}, there exists a unique function $Y_n\in L^2([0,T],V)\cap
C([0,T],H)$
such that
    \begin{eqnarray}\label{Solution-Y_n}
        \langle Y_{n}(t),\phi\rangle
     &=&
        \langle X_0,\phi\rangle
        -
          \int_0^t\langle \mathcal{A}Y_{n}(s),\phi\rangle ds
        +
          \int_0^t\langle\sigma(s,Y_{n-1}(s))f(s),\phi\rangle ds\nonumber\\
       & &+
        \int_0^t\int_{\mathbb{X}}\langle G(s,Y_{n-1}(s),v)(g(s,v)-1),\phi\rangle\nu(dv)ds,\ \phi\in V;
    \end{eqnarray}
and
\begin{eqnarray*}
   &  & \|Y_{n+1}(t)-Y_n(t)\|_H^2\nonumber\\
   & = &
            -2\int_0^t \langle \mathcal{A}(Y_{n+1}(s)-Y_n(s)), Y_{n+1}(s)-Y_n(s)\rangle ds\\
         &&   +2 \int_0^t\langle(\sigma(s,Y_{n}(s))-\sigma(s,Y_{n-1}(s)))f(s),Y_{n+1}(s)-Y_n(s)\rangle ds\\
       &   &+
            2\int_0^t\int_{\mathbb{X}}\langle
            G(s,Y_n(s),v)-G(s,Y_{n-1}(s),v),Y_{n+1}(s)-Y_n(s)\rangle(g(s,v)-1)\nu(dv)ds.\nonumber
\end{eqnarray*}
In view of $(\ref{condition-lower bounds of A})$,
\begin{eqnarray*}
&  & \|Y_{n+1}(t)-Y_n(t)\|_H^2+\alpha\int_0^t\|Y_{n+1}(s)-Y_n(s)\|^2_Vds\\
   &\leq&
         2 \int_0^t\|(\sigma(s,Y_{n}(s))-\sigma(s,Y_{n-1}(s)))f(s)\|_H\|Y_{n+1}(s)-Y_n(s)\|_H ds\\
       &  &+
         2\int_0^t\int_{\mathbb{X}}\|G(s,Y_n(s),v)-G(s,Y_{n-1}(s),v)\|_H\|Y_{n+1}(s)-Y_n(s)\|_H|g(s,v)-1|\nu(dv)ds \\
         &&+\lambda_{0}\int_{0}^{t}\|Y_{n+1}(s)-Y_n(s)\|^2_{H}ds\\
   &\leq&
         2\int_{0}^{t}\sqrt{K(s)}\|Y_{n}(s)-Y_{n-1}(s)\|_H\|Y_{n+1}(s)-Y_n(s)\|_H   \|f(s)\|_{H}\;ds\\
         &&+   \lambda_{0}\int_{0}^{t}\|Y_{n+1}(s)-Y_n(s)\|^2_{H}ds\\
         &&+
         \int_0^t\int_{\mathbb{X}}\frac{\|G(s,Y_n(s),v)-G(s,Y_{n-1}(s),v)\|_H}{\|Y_n(s)-Y_{n-1}(s)\|_H}\|Y_n(s)-Y_{n-1}(s)\|_H\\
         &&\quad\quad \cdot \|Y_{n+1}(s)-Y_n(s)\|_H|g(s,v)-1|\nu(dv)ds\\
   &\leq&
         C\int_0^t{K(s)}\|Y_{n+1}(s)-Y_n(s)\|^{2}_H\;ds\\
         &&+C\int_0^t\|Y_{n}(s)-Y_{n-1}(s)\|^{2}_H\|f(s)\|^{2}_H ds\\
         &&+\lambda_{0}\int_{0}^{t}\|Y_{n+1}(s)-Y_n(s)\|^2_{H}ds\\
         &   &+
         \int_0^t \left(\int_{\mathbb{X}}\|G(s,v)\|_{1,H}|g(s,v)-1|\nu(dv)\right)\|Y_n(s)-Y_{n-1}(s)\|_H \|Y_{n+1}(s)-Y_n(s)\|_{H}ds\\
    &\leq&
         C\int_0^t{K(s)}\|Y_{n+1}(s)-Y_n(s)\|^{2}_H\;ds\\
         &&+C\int_0^t\|Y_{n}(s)-Y_{n-1}(s)\|^{2}_H\|f(s)\|^{2}_H ds\\
         &&+\lambda_{0}\int_{0}^{t}\|Y_{n+1}(s)-Y_n(s)\|^2_{H}ds\\
         &   &+\int_0^t \left(\int_{\mathbb{X}}\|G(s,v)\|_{1,H}|g(s,v)-1|\nu(dv)\right)\|Y_n(s)-Y_{n-1}(s)\|^{2}_H\;ds\\
         &&+\int_0^t \left(\int_{\mathbb{X}}\|G(s,v)\|_{1,H}|g(s,v)-1|\nu(dv)\right) \|Y_{n+1}(s)-Y_n(s)\|^{2}_{H}\;ds
\end{eqnarray*}

We denote $J(s)=K(s)+\|f(s)\|^2_H+\int_{\mathbb{X}}\|G(s,v)\|_{1,H}|g(s,v)-1|\nu(dv)+\lambda_{0}$, and set $a_{n}(t)=\|Y_{n}(t)-Y_{n-1}(t)\|_H^2$.

The above inequality yields that
\begin{eqnarray}\label{3.17}
\label{relation.an and an+1 (1)}
a_{n+1}(t)\leq C\int_{0}^{t} a_{n}(s) J(s)\;ds+ C\int_{0}^{t}a_{n+1}(s) J(s)\;ds,
\end{eqnarray}
and furthermore, we have
\begin{eqnarray*}
a_{n+1}(t)e^{-C\int_{0}^{t}J(s)\;ds}J(t)&\leq& C e^{-C\int_{0}^{t}J(s)\;ds}J(t)\int_{0}^{t} a_{n}(s) J(s)\;ds\\
&&+C e^{-C\int_{0}^{t}J(s)\;ds}J(t)\int_{0}^{t}a_{n+1}(s) J(s)\;ds
\end{eqnarray*}

Set $A_{n+1}(t)=\int_{0}^{t}a_{n+1}(s)J(s)\;ds$. It follows that,
\begin{eqnarray*}
\frac{d}{dt}\left( A_{n+1}(t)e^{-C\int_{0}^{t}J(s)\;ds}\right)\leq C J(t)e^{-C\int_{0}^{t}J(s)\;ds}A_{n}(t).
\end{eqnarray*}

So that,
\begin{eqnarray*}
 A_{n+1}(t)e^{-C\int_{0}^{t}J(s)\;ds}&\leq& C  \int_{0}^{t} J(s)e^{-C\int_{0}^{s}J(u)\;du}A_{n}(s)\;ds.
\end{eqnarray*}

Thus,
\begin{eqnarray*}
 A_{n+1}(t)&\leq& Ce^{C\int_{0}^{t}J(s)\;ds}\int_{0}^{t} J(s)A_{n}(s)\;ds\\
 &\leq& Ce^{C\int_{0}^{t}J(s)\;ds} A_{n}(t)\int_{0}^{t} J(s)\;ds\\
 &\leq& C_{T} A_{n}(t).
\end{eqnarray*}

It follows from (\ref{3.17}) that
\begin{eqnarray}
\label{relation.an and an+1 (2)}
a_{n+1}(t)\leq (C+C_{T})\int_{0}^{t} a_{n}(s) J(s)\;ds.
\end{eqnarray}

Iterating the above inequality, we get
\begin{eqnarray}
\label{relation.an and an+1}
a_{n+1}(t)\leq \frac{(C+C_{T})^{n}(\int_{0}^{T}J(s)\;ds)^{n}}{n!}\times \sup_{s\in[0,T]}a_1(s).
\end{eqnarray}

Therefore, we have
\begin{eqnarray*}
\sum_{n=0}^{\infty}a_{n+1}(t)<\infty.
\end{eqnarray*}
%

Hence there exists $Y\in C([0,T],H)$ such that $\lim_{n\rightarrow\infty}\sup_{s\in[0,T]}\|Y(s)-Y_n(s)\|^2_H=0$.

On the other hand, by Lemma \ref{Lemma-thm-01} and Lemma \ref{Lemma-thm-02}, there exists a unique function $Y'\in
L^2([0,T],V)\cap C([0,T],H)$
such that
    \begin{eqnarray}\label{Solution-Y'}
        \langle Y'(t),\phi\rangle
     &=&
        \langle X_0,\phi\rangle
        -
        \int_0^t\langle \mathcal{A}Y'(s),\phi\rangle ds
        +
        \int_0^t\langle\sigma(s,Y(s))f(s),\phi\rangle ds\nonumber\\
       & &+
        \int_0^t\int_{\mathbb{X}}\langle G(s,Y(s),v)(g(s,v)-1),\phi\rangle\nu(dv)ds,\ \phi\in V.
    \end{eqnarray}

Using the same argument  leading to  (\ref{relation.an and an+1}),
we have
   \begin{eqnarray}\label{Y'-Y}
  \lim_{n\rightarrow\infty} \sup_{s\in[0,T]}\|Y'(s)-Y_n(s)\|^2_H=0.
\end{eqnarray}

Hence $Y'=Y,\ t\in[0,T]$ is a solution to (\ref{main-thm-01-eq}).

We have proved the existence of the solution.

{\bf Uniqueness}:
Assume $X$ and $X'$ are two solutions of equation (\ref{main-thm-01-eq}).  Then, as the proof of (\ref{relation.an and an+1 (2)}),
we have,
\begin{eqnarray}\label{X'-X}
    \sup_{s\in[0,T]}\|X'(s)-X(s)\|^2_H
& \leq&
    (C+C_{T})\int_{0}^{T}\|X'(s)-X(s)\|^2_H J(s)\;ds \nonumber
\end{eqnarray}
By Gronwall's inequality, we conclude $X'=X$.

Finally we prove the estimate  (\ref{main-thm-01-ineq}). By Lemma \ref{Lemma-thm-01}, we have
    \begin{eqnarray}\label{proof-main-thm-01-02-eq-01}
     &  &\|\widetilde{X}^q_t\|^2_H+\alpha\int_0^t\|\widetilde{X}^q_s\|^2_Vds\nonumber\\
     &\leq&
          \|X_0\|^2_H
            +
            2\int_0^t\|\sigma(s,\widetilde{X}^q_s)\|_{L_2(H)}\|f(s)\|_H\|\widetilde{X}^q_s\|_H ds\nonumber\\
           & &+
          2\int_0^t\int_\mathbb{X}\|G(s,\widetilde{X}^q_s,v)\|_H\|\widetilde{X}^q_s\|_H|g(s,v)-1|\nu(dv)ds\nonumber
          +\lambda_{0}\int_{0}^{t}\|\widetilde{X}^q_s\|_{H}^{2}\;ds\\
     &=&
          \|X_0\|^2_H
            +
            2\int_0^t\sqrt{K(s)}\|f(s)\|_H\|\widetilde{X}^q_s\|^2_H ds\nonumber\\
            & &+
          2\int_0^t\int_\mathbb{X}\frac{\|G(s,\widetilde{X}^q_s,v)\|_H}{1+\|\widetilde{X}^q_s\|_H}(1+\|\widetilde{X}^q_s\|_H)\|\widetilde{X}^q_s\|_H|g(s,v)-1|\nu(dv)ds\nonumber\\
          &&+\lambda_{0}\int_{0}^{t}\|\widetilde{X}^q_s\|_{H}^{2}\;ds\\
     &\leq&
          \|X_0\|^2_H
            +
          C\int_0^t\int_\mathbb{X}\|G(s,v)\|_{0,H}|g(s,v)-1|\nu(dv)ds\nonumber\\
            &   &+
          2\int_0^t\|\widetilde{X}^q_s\|^2_H\Big[\lambda_{0}+K(s)+\|f(s)\|^2_H+\int_\mathbb{X}\|G(s,v)\|_{0,H}|g(s,v)-1|\nu(dv)\Big]ds
             \end{eqnarray}
By Gronwall's inequality,
        \begin{eqnarray}\label{proof-main-thm-01-02-eq-02}
             &  & \sup_{s\in[0,t]}\|\widetilde{X}^q_s\|^2_H\nonumber\\
             &\leq&
              C\Big[\|X_0\|^2_H
              +
             \int_0^t\int_\mathbb{X}\|G(s,v)\|_{0,H}|g(s,v)-1|\nu(dv)ds\Big]\\
              &   &\times
             \exp\big(C\Big[\lambda_{0}+\int_0^tK(s)ds+\int_0^t\|f(s)\|^2_Hds+\int_0^t\int_\mathbb{X}\|G(s,v)\|_{0,H}|g(s,v)-1|\nu(dv)ds\Big]\Big).\nonumber
             \end{eqnarray}
By Lemma \ref{Lemma-Condition-0,H-1,H},  (\ref{proof-main-thm-01-02-eq-01}) and  (\ref{proof-main-thm-01-02-eq-02}),
we obtain (\ref{main-thm-01-ineq}).

\end{proof}

\vspace{3mm}

We can now present the main large deviations result. Recall that for $q=(f,g)\in\mathbb{S}$,
$\nu_T^g(dsdv)=g(s,v)\nu(dv)ds$.
Define
     \begin{eqnarray}\label{LDP-eq-01}
       \mathcal{G}^0(\int_0^\cdot f(s)ds,\nu^g_T)=\widetilde{X}^q\ \text{for}\ q=(f,g)\in\mathbb{S}\text{\ as\  given\ in\ Theorem}\
       \ref{Main-thm-01}.
     \end{eqnarray}
Let $I:D([0,T],H)\rightarrow [0,\infty]$ be defined as in (\ref{Rate function I}).

\begin{thm}\label{Main-thm-02}
 Suppose that Condition \ref{Condition G} and Condition \ref{Condition Exponential Integrability}
 hold. Then $I$ is a rate function on $D([0,T],H)$, and the family $\{X^\epsilon\}_{\epsilon>0}$ satisfies
 a large deviation principle on $D([0,T],H)$ with rate function $I$.
\end{thm}

The rest of the paper is devoted to the proof of this theorem. According to Theorem \ref{LDP-main-01}, we need to prove that Condition \ref{LDP} is fulfilled. The verification of Condition \ref{LDP} a) will be given by Proposition \ref{Prop-1}. Condition \ref{LDP} b) will be established in Theorem \ref{Prop-2} and a number of preparing lemmas.


Let $T_t,t\geq0$ denote the semigroup generated by $-\mathcal{A}$. It is easy to see that $T_t,t\geq0$ are
    compact operators. For $f\in L^1([0,T],H)$, denote the operator
    \begin{eqnarray*}
        Rf(t)=\int_0^tT_{t-s}f(s)ds,\ t\geq0,
    \end{eqnarray*}
    which is the mild solution of the equation:
    \begin{eqnarray*}
       Z(t)=-\int_0^t\mathcal{A}Z(s)ds+\int_0^tf(s)ds.
    \end{eqnarray*}

The proof of the following lemma was given in \cite{Rockner-Zhang}.
\begin{lem}\label{lem-thm2-01}
    If $\mathcal{D}\subset L^1([0,T],H)$ is uniformly integrable, then $\mathcal{Y}=R(\mathcal{D})$ is
    relatively compact in $C([0,T],H)$.
\end{lem}

We also need the following lemma, the proof of which can be found in \cite{Budhiraja-Chen-Dupuis}.
\begin{lem}\label{lem-thm2-02}
Let
$h:[0,T]\times\mathbb{X}\rightarrow\mathbb{R}$ be a measurable function such that
   \begin{eqnarray*}
     \int_{\mathbb{X}_T}|h(s,v)|^2\nu(dv)ds<\infty,
   \end{eqnarray*}
and for all $\delta\in(0,\infty)$
   \begin{eqnarray*}
     \int_{E}\exp(\delta|h(s,v)|)\nu(dv)ds<\infty,
   \end{eqnarray*}
for all $E\in\mathcal{B}([0,T]\times\mathbb{X})$ satisfying $\nu_T(E)<\infty$.

a). Fix $N\in\mathbb{N}$, and let $g_n,g\in S^N$ be such that $g_n\rightarrow g$ as $n\rightarrow\infty$. Then
   \begin{eqnarray*}
     \lim_{n\rightarrow\infty}\int_{\mathbb{X}_T}h(s,v)(g_n(s,v)-1)\nu(dv)ds=\int_{\mathbb{X}_T}h(s,v)(g(s,v)-1)\nu(dv)ds;
   \end{eqnarray*}

b). Fix $N\in\mathbb{N}$. Given $\epsilon>0$, there exists a compact set $K_\epsilon\subset\mathbb{X}$, such that
    \begin{eqnarray*}
      \sup_{g\in S^N}\int_{[0,T]}\int_{K_\epsilon^c}|h(s,v)||g(s,v)-1|\nu(dv)ds\leq\epsilon.
    \end{eqnarray*}
\end{lem}

\vspace{2mm}

We now proceed to verify the first part of Condition \ref{LDP}. Recall the map $\mathcal{G}^0$ defined by
(\ref{LDP-eq-01}).

\begin{prp}\label{Prop-1}
 Fix $N\in\mathbb{N}$, and let $q_n=(f_n,g_n),q=(f,g)\in \bar{S}^N$ be such that $q_n\rightarrow q$ as $n\rightarrow\infty$. Then
     \begin{eqnarray*}
       \mathcal{G}^0(\int_0^\cdot f_n(s)ds,\nu_T^{g_n})\rightarrow\mathcal{G}^0(\int_0^\cdot f(s)ds,\nu_T^{g})\quad \text{in}\quad C([0,T],H).
     \end{eqnarray*}
\end{prp}

\begin{proof}

Firstly, we prove that $\{\mathcal{G}^0(\int_0^\cdot f_n(s)ds,\nu_T^{g_n})\}_{n\in\mathbb{N}}$ is relatively compact in $C([0,T],H)$.

By Theorem \ref{Main-thm-01} and the relation between mild solution and weak solution,
\begin{eqnarray}
&  &  \mathcal{G}^0(\int_0^\cdot f_n(s)ds,\nu_T^{g_n})(t)\\
& =&
  T(t)X_0
+
\int_0^tT(t-s)\Big[\sigma(s,\widetilde{X}^{q_n}(s))f_n(s)+\int_{\mathbb{X}}G(s,\widetilde{X}^{q_n}(s),v)(g_n(s,v)-1)\nu(dv)\Big]ds.\nonumber
\end{eqnarray}

By Lemma \ref{lem-thm2-01}, it is sufficient to prove that
$$
  \mathcal{D}=\Big\{\sigma(\cdot,\widetilde{X}^{q_n}(\cdot))f_n(\cdot)+\int_{\mathbb{X}}G(\cdot,\widetilde{X}^{q_n}(\cdot),v)(g_n(\cdot,v)-1)\nu(dv)\Big\}\subset
  L^1([0,T],H)
$$
is uniformly integrable.

We know that $\mathcal{D}$ is uniformly integrable in $L^1([0,T],H)$ iff

(I) There exists a finite constant $\widehat{K}$ such that,
$\int_0^T\|h(s)\|_Hds\leq\widehat{K}$, for every $h\in\mathcal{D}$,;

(II) For every $\eta>0$ there exists $\delta>0$ such that, for every measurable subset $A\subset[0,T]$ with
$\lambda_T(A)\leq\delta$
and every $h\in\mathcal{D}$, $\int_A\|h(s)\|_Hds\leq\eta$.

In fact (I) follows from Lemma \ref{Lemma-thm-02} and Theorem \ref{Main-thm-01}.
%
%
We need to check (II).
For any $n\in\mathbb{N}$ and any measurable subset $A\subset[0,T]$,
   \begin{eqnarray}\label{pf-main-thm-02-UI-02-02}
          &  & \int_A\|\int_{\mathbb{X}}G(s,\widetilde{X}^{q_n}(s),v)(g_n(s,v)-1)\nu(dv)\|_{H}ds\nonumber\\
    &\leq&
          \int_A\int_{\mathbb{X}}\|G(s,\widetilde{X}^{q_n}(s),v)\|_{H}|g_n(s,v)-1|\nu(dv)ds\nonumber\\
    &\leq&
          \int_A\int_{\mathbb{X}}\frac{\|G(s,\widetilde{X}^{q_n}(s),v)\|_{H}}{1+\|\widetilde{X}^{q_n}(s)\|_H}(1+\|\widetilde{X}^{q_n}(s)\|_H)|g_n(s,v)-1|\nu(dv)ds\nonumber\\
    &\leq&
          (1+\sup_{s\in[0,T]}\|\widetilde{X}^{q_n}(s)\|_H)\int_A\int_{\mathbb{X}}\|G(s,v)\|_{0,H}|g_n(s,v)-1|\nu(dv)ds.\nonumber
  \end{eqnarray}
  Given $\epsilon>0$, by Lemma \ref{lem-thm2-02} we can find a compact subset $K_\epsilon\subset \mathbb{X}$ such that
  \begin{eqnarray}\label{T2-1}
  \int_A\int_{K_\epsilon^c}\|G(s,v)\|_{0,H}|g_n(s,v)-1|\nu(dv)ds\leq\epsilon,\ \forall n\geq1.
  \end{eqnarray}
On the other hand, by  (\ref{Inequality-001}) for any $L>0$, we have
   \begin{eqnarray}\label{T2-2}
       &   &\int_A\int_{K_\epsilon}\|G(s,v)\|_{0,H}|g_n(s,v)-1|\nu(dv)ds\nonumber\\
  &\leq&
        C(L)\int_A\int_{K_\epsilon}\exp(L\|G(s,v)\|_{0,H})\nu(dv)ds+\frac{1}{L}\int_0^T\int_{\mathbb{X}}l(g_n(s,v))\nu(dv)ds\nonumber\\
      &   &+\int_A\int_{K_\epsilon}\exp(\|G(s,v)\|_{0,H})\nu(dv)ds.
  \end{eqnarray}
We also have
\begin{eqnarray}
\label{T2-3}
&   & \int_A\|\sigma(s,\widetilde{X}^{q_n}(s))f_n(s)\|_Hds\nonumber\\
&\leq&
     \int_A\|\sigma(s,\widetilde{X}^{q_n}(s))\|_{L_2(H)}\|f_n(s)\|_Hds\nonumber\\
&\leq&
     \int_A\sqrt{K(s)}\sqrt{\|\widetilde{X}^{q_n}(s)\|^{2}_{H}+1}\|f_n(s)\|_Hds\nonumber\\
&\leq&
     \sqrt{\sup_{s\in[0,T]}\|\widetilde{X}^{q_n}(s)\|^{2}_{H}+1}\times\Big[(\int_AK(s)ds)(\int_0^T\|f_n(s)\|^2_Hds)\Big]^{1/2}\nonumber\\
&\leq&[\int_AK(s)ds]^{1/2}\times (1+C_N)^{1/2}N^{1/2},
\end{eqnarray}
where $\sup_{s\in[0,T]}\|\widetilde{X}^{q_n}(s)\|^{2}_{H}\leq C_{N}$ and $\int_0^T\|f_n(s)\|^2_H\;ds\leq N$.

Now for any $\eta>0$, first choose $\epsilon>0$ such that $(1+C_N)\epsilon\leq\eta/5$, then select $L>0$ so that
$(1+C_N)N/L\leq\eta/5.$
Finally since $\nu_T([0,T]\times K_\epsilon)=T\nu(K_\epsilon)<\infty$, there exists $\delta>0$ such that for every
measurable subset $A\subset[0,T]$ satisfying $\lambda_T(A)\leq\delta$,
we have
$$
(1+C_N)C(L)\int_A\int_{K_\epsilon}\exp(L\|G(s,v)\|_{0,H})\nu(dv)ds\leq\eta/5,
$$
$$
(1+C_N)\int_A\int_{K_\epsilon}\exp(\|G(s,v)\|_{0,H})\nu(dv)ds\leq\eta/5,
$$
and
$$
[\int_AK(s)ds]^{1/2}\times C_N^{1/2}N^{1/2}\leq\eta/5.
$$
Combining all these inequalities gives (II).

\vspace{3mm}
Let $\widetilde{X}$ be any limit point of the sequence $\{\widetilde{X}^{q_n},n\geq1\}$. Now we will prove that
$\widetilde{X}=\widetilde{X}^q$. Without
loss of generality, we assume the whole sequence $\{\widetilde{X}^{q_n}\}$ converges.

An application of dominated convergence theorem shows that, along the convergent subsequence,
    \begin{eqnarray}\label{pf-main-thm-02-limit-01}
      \int_0^t\langle \widetilde{X}^{q_n}(s),\mathcal{A}^*\phi\rangle ds
            \rightarrow
               \int_0^t\langle \widetilde{X}(s),\mathcal{A}^*\phi\rangle ds,\ \forall\ \phi\in H_{\mathcal{A}^*},
    \end{eqnarray}
as $n\rightarrow\infty$. Furthermore, using the convergence of $\widetilde{X}^{q_n}$ to $\widetilde{X}$, and Lemma
\ref{Lemma-Condition-0,H-1,H}, we have that
    \begin{eqnarray}\label{pf-main-thm-02-limit-02}
     & &\int_0^t\int_{\mathbb{X}}\langle G(s,\widetilde{X}^{q_n}(s),v),\phi\rangle(g_n(s,v)-1)\nu(dv)ds\nonumber\\
      &-&
      \int_0^t\int_{\mathbb{X}}\langle G(s,\widetilde{X}(s),v),\phi\rangle(g_n(s,v)-1)\nu(dv)ds
     \rightarrow
      0,\ \forall\ \phi\in V.
    \end{eqnarray}

Also, since $\widetilde{X}\in C([0,T],H)$, we have for some $\kappa\in(0,\infty)$
$$
 |\langle G(s,\widetilde{X}(s),v),\phi\rangle|\leq\kappa\|\langle G(s,v),\phi\rangle\|_{0,H},\ \forall\phi\in V,\
 \forall(s,v)\in \mathbb{X}_T.
$$

Combining this with Condition \ref{Condition G} and Lemma \ref{lem-thm2-02}, we have, as $n\rightarrow\infty$,
\begin{eqnarray}\label{pf-main-thm-02-limit-03}
    &  &\int_0^t\int_{\mathbb{X}}\langle G(s,\widetilde{X}(s),v),\phi\rangle(g_n(s,v)-1)\nu(dv)ds\nonumber\\
  & \rightarrow&
      \int_0^t\int_{\mathbb{X}}\langle G(s,\widetilde{X}(s),v),\phi\rangle(g(s,v)-1)\nu(dv)ds,\ \forall\phi\in V.
\end{eqnarray}

We also have, for every $\phi\in V$,
\begin{eqnarray}\label{pf-main-thm-02-limit-05}
&    &|\int_0^t\langle \sigma(s,\widetilde{X}^{q_n}(s))f_n(s),\phi\rangle ds-\int_0^t\langle \sigma(s,\widetilde{X}(s))f(s),\phi\rangle ds|\nonumber\\
&\leq&
    \|\phi\|_H
        \Big[
            \int_0^t\|\sigma(s,\widetilde{X}^{q_n}(s))-\sigma(s,\widetilde{X}(s))\|_{L_2(H)}\|f_n(s)\|_Hds\nonumber\\
          &&+
            \int_0^t\|\sigma(s,\widetilde{X}(s))\|_{L_2(H)}\|f_n(s)-f(s)\|_Hds
        \Big]\nonumber\\
&\leq&
    \|\phi\|_H 
        \Big[
            \int_0^t\sqrt{K(s)}\|\widetilde{X}^{q_n}(s)-\widetilde{X}(s)\|_{H}\|f_n(s)\|_Hds\nonumber\\
          &  &\ \ \ \ \ \ \ \ +
            \int_0^t\sqrt{K(s)}\sqrt{\|\widetilde{X}(s)\|^{2}_{H}+1}\|f_n(s)-f(s)\|_Hds
        \Big]\nonumber\\
&\leq&
    \|\phi\|_H
         \Big[
               \sup_{s\in[0,T]}\|\widetilde{X}^{q_n}(s)-\widetilde{X}(s)\|_{H}\times \Big( \int_0^TK(s)ds+\int_0^T\|f_n(s)\|^2_Hds \Big)\nonumber\\
             &  &\ \ \ \ \ \ \ \ +
               \sqrt{\sup_{s\in[0,T]}\|\widetilde{X}(s)\|^{2}_{H}+1}\times(\int_0^TK(s)ds)^{\frac{1}{2}}(\int_0^T\|f_n(s)-f(s)\|^2_Hds )^{1/2}
         \Big]\nonumber\\
&\rightarrow& 0,\ as\ n\rightarrow\infty.\nonumber \\
\end{eqnarray}

Combining (\ref{pf-main-thm-02-limit-01}), (\ref{pf-main-thm-02-limit-02}), (\ref{pf-main-thm-02-limit-03}) and  (\ref{pf-main-thm-02-limit-05}),
we  see that $\widetilde{X}$ must satisfy
\begin{eqnarray}\label{pf-main-thm-02-limit-04}
       \langle \widetilde{X}_t,\phi\rangle
          &=&
            \langle X_0,\phi\rangle
          -
            \int_0^t\langle \mathcal{A}\widetilde{X}_s,\phi\rangle ds
          +
            \int_0^t\langle\sigma(s,\widetilde{X}(s))f(s),\phi\rangle ds\nonumber\\
          & &+
            \int_0^t\int_{\mathbb{X}}\langle G(s,\widetilde{X}_s,v),\phi\rangle(g(s,v)-1)\nu(dv)ds,\ \forall \phi\in
            H_{\mathcal{A}^*}.
\end{eqnarray}


Since $H_{\mathcal{A}^*}$ is dense in $V$, we have $\widetilde{X}=\widetilde{X}^q$, completing the proof.
\end{proof}
\begin{remark}\label{Remark-03}
Fix $N\in\mathbb{N}$. By the proof of (\ref{pf-main-thm-02-UI-02-02}), it is easy to see that, for ever $\eta>0$,
there exists $\delta>0$
such that for any $A\subset[0,T]$ satisfying $\lambda_T(A)<\delta$
\begin{eqnarray}\label{G_0H-inequation}
\sup_{g\in S^N}\int_A\int_{\mathbb{X}}\|G(s,v)\|_{0,H}|g(s,v)-1|\nu(dv)ds\leq \eta.
\end{eqnarray}
\end{remark}

We now verify the second part of Condition $\ref{LDP}$. The next theorem can be proved similarly as in Section 3 in
\cite{Rockner-Zhang}.
\begin{thm}
\label{Thm-solution with random}
Under Condition $\ref{Condition G}$, if $X_{0}\in H$,  there exists a unique $H$-valued progressively
measurable process $(X^{\epsilon}_{t})$ such that, $X^{\epsilon}\in L^{2}(0,T; V)\cap D([0,T]; H)$ for any $T>0$ and
\begin{eqnarray}
\label{Equi-spde}
X^{\epsilon}_{t}=X^{\epsilon}_{0}-\int_{0}^{t}\mathcal{A}X^{\epsilon}_{s}ds+\sqrt{\epsilon}\int_{0}^{t}\sigma(s,X^{\epsilon}_{s})d\beta(s)+\epsilon
\int_{0}^{t}\int_{\mathbb{X}}G(s,X^{\epsilon}_{s},v)\tilde{N}^{\epsilon^{-1}}(dsdv), a.s..
\end{eqnarray}
\end{thm}
\vspace{3mm}

It follows from this theorem that, for every $\epsilon>0$, there exists a measurable map $\mathcal{G}^{\epsilon}$:
$\mathbb{V}\rightarrow D([0,T]; H)$ such that, for any Poisson random measure $n^{\epsilon^{-1}}$ on $[0,T]\times
\mathbb{X}$ with mean measure $\epsilon^{-1} \lambda_{T}\times v$ given on some probability space,
$\mathcal{G}^{\epsilon}(\sqrt{\epsilon}\beta, \epsilon n^{\epsilon^{-1}})$ is the unique solution of
$(\ref{Equi-spde})$ with
$\tilde{N}^{\epsilon^{-1}}$ replaced by $\tilde{n}^{\epsilon^{-1}}$.

Let $\phi_{\epsilon}=(\psi_{\epsilon}, \varphi_{\epsilon})\in \tilde{\mathcal{U}}^N$ and $\phi_{\epsilon}=\frac{1}{\varphi_{\epsilon}}\in \tilde{\mathcal{U}}^{N}$. The following
lemma follows from Lemma 2.3 in \cite{Budhiraja-Dupuis-Maroulas.}.
\begin{lem}
\begin{eqnarray*}
\mathcal{E}^{\epsilon}_{t}(\phi_{\epsilon}):=\exp\big\{&\int_{[0,t]\times \mathbb{X}\times
[0,\epsilon^{-1}]}&\log(\phi_{\epsilon}(s,x))\bar{N}(\;ds\;dx\;dr)\\
&&+\int_{[0,t]\times \mathbb{X}\times
[0,\epsilon^{-1}]}(-\phi_{\epsilon}(s,x)+1)\bar{\nu}_{T}(\;ds\;dx\;dr)\big\}
\end{eqnarray*}
and
$$
\tilde{\mathcal{E}}^{\epsilon}_{t}(\psi_{\epsilon})
:=\exp\{\frac{1}{\sqrt{\epsilon}}\int_{0}^{t}\psi_{\epsilon}(s)\;d\beta(s)-\frac{1}{2\epsilon}\int_{0}^{t}\|\psi_{\epsilon}(s)\|^{2}\;ds\}
$$
are $\{\bar{\mathcal{F}}_{t}\}-$ martingales. Set
$$
\bar{\mathcal{E}}^{\epsilon}_{t}(\psi_{\epsilon},\phi_{\epsilon}):=\tilde{\mathcal{E}}^{\epsilon}_{t}(\psi_{\epsilon})\mathcal{E}^{\epsilon}_{t}(\phi_{\epsilon})
$$
then
$$
\mathbb{Q}^{\epsilon}_{t}(G)=\int_{G}\bar{\mathcal{E}}^{\epsilon}_{t}(\psi_{\epsilon},\phi_{\epsilon})\;d\bar{\mathbb{P}},
\quad
\text{for}\quad G\in \mathcal{B}(\bar{\mathbb{V}})
$$
defines a probability measure on $\bar{\mathbb{V}}$.
\end{lem}

By the fact that $\epsilon N^{\epsilon^{-1}\varphi_{\epsilon}}$ under $\mathbb{Q}^{\epsilon}_{T}$ has the same law as
that of $\epsilon N^{\epsilon^{-1}}$ under $\bar{\mathbb{P}}$, we know that
$\tilde{X}^{\epsilon}=\mathcal{G}^{\epsilon}(\sqrt{\epsilon}\beta+\int_{0}^{\cdot}\psi_{\epsilon}(s)\;ds,\epsilon
N^{\epsilon^{-1}\varphi_{\epsilon}})$ is the unique solution of
the following controlled stochastic differential equation:
\begin{eqnarray}
\tilde{X}^{\epsilon}_{t}&=&X_{0}-\int_{0}^{t}\mathcal{A}\tilde{X}^{\epsilon}_{s}ds
+\int_{0}^{t}\sigma(s,\tilde{X}^{\epsilon}_{s})\varphi_{\epsilon}(s)ds
+\sqrt{\epsilon}\int_{0}^{t}\sigma(s,X^{\epsilon}_{s})d\beta(s)\nonumber\\
&&+\int_{0}^{t}\int_{\mathbb{X}}G(s,\tilde{X}^{\epsilon}_{s-},v)\left(\epsilon N^{\epsilon^{-1}\varphi_{\epsilon}}(ds
dv)-\nu(dv)ds\right)\nonumber\\
&=&X_{0}-\int_{0}^{t}\mathcal{A}\tilde{X}^{\epsilon}_{s}ds+\int_{0}^{t}\sigma(s,\tilde{X}^{\epsilon}_{s})\varphi_{\epsilon}(s)ds
+\sqrt{\epsilon}\int_{0}^{t}\sigma(s,X^{\epsilon}_{s})d\beta(s)\nonumber\\
&&+\int_{0}^{t}\int_{\mathbb{X}}G(s,\tilde{X}^{\epsilon}_{s},v)
(\varphi_{\epsilon}(s,v)-1)\nu(dv)ds\nonumber\\
&&+\int_{0}^{t}\int_{\mathbb{X}}\epsilon G(s,\tilde{X}^{\epsilon}_{s-},v)\left(
N^{\epsilon^{-1}\varphi_{\epsilon}}(ds
dv)-\epsilon^{-1}\varphi_{\epsilon}(s,v)\nu(dv)ds\right)
\end{eqnarray}

The following estimates will be used later.
\begin{lem}\label{Lemma-estimation-SPDE}
There exists $\epsilon_{0}>0$, such that,
\begin{eqnarray}\label{estimation-01}
\sup_{0<\epsilon<\epsilon_{0}}\mathbb{E}\sup_{0\leq t\leq T}\|\tilde{X}^{\epsilon}_{t}\|_{H}^{2}<\infty.
\end{eqnarray}
\end{lem}
\begin{proof}
By Ito's formula,
\begin{eqnarray}
\label{Ito formula}
&&\|\tilde{X}^{\epsilon}_{t}\|_{H}^{2}\nonumber\\
&=&\|X_{0}\|_{H}^{2}
-2\int_{0}^{t}\langle \tilde{X}^{\epsilon}_{s}, \mathcal{A}^*\tilde{X}^{\epsilon}_{t}\rangle ds
+2\int_{0}^{t}\langle \tilde{X}^{\epsilon}_{s}, \sigma(s,\tilde{X}^{\epsilon}_{t})\psi_{\epsilon}(s)\rangle
\;ds+2\sqrt{\epsilon}\int_{0}^{t}\sigma(s,\tilde{X}^{\epsilon}_{s})\tilde{X}^{\epsilon}_{s}\;d\beta_{s}\nonumber\\
&+&2\int_{0}^{t}\int_{\mathbb{X}}<\tilde{X}^{\epsilon}_{s},G(s,\tilde{X}^{\epsilon}_{s},v)>(\varphi_{\epsilon}(s,v)-1)\nu(dv)ds\nonumber\\
&+&\int_{0}^{t}\int_{\mathbb{X}}\left(2<\epsilon
G(s,\tilde{X}^{\epsilon}_{s-},v),\tilde{X}^{\epsilon}_{s}>+\|\epsilon
G(s,\tilde{X}^{\epsilon}_{s-},v)\|_{H}^{2}\right)\left( N^{\epsilon^{-1}\varphi_{\epsilon}}(ds
dv)-\epsilon^{-1}\varphi_{\epsilon}(s,v)\nu(dv)ds\right)\nonumber\\
&+&\epsilon\int_{0}^{t}\int_{\mathbb{X}}\|G(s,\tilde{X}^{\epsilon}_{s},v)\|_{H}^{2}\varphi_{\epsilon}(s,v)\nu(dv)ds
+\epsilon\int_{0}^{t}\|\sigma(s,\tilde{X}^{\epsilon}_{s})\|_{L_{2}(H)}^{2}\;ds.
\end{eqnarray}

The third term on right hand side of last equation is estimated as follows.
\begin{eqnarray}
\label{estimates-third term}
2\int_{0}^{t}\langle \tilde{X}^{\epsilon}_{s}, \sigma(s,\tilde{X}^{\epsilon}_{t})\psi_{\epsilon}(s)\rangle \;ds
&\leq& 2\int_{0}^{t}\| \tilde{X}^{\epsilon}_{s}\|_{H}
\|\sigma(s,\tilde{X}^{\epsilon}_{t})\|_{L_{2}(H)}\|\psi_{\epsilon}(s)\|_{H} \;ds\nonumber\\
&\leq& \int_{0}^{t}\| \tilde{X}^{\epsilon}_{s}\|^{2}_{H}\|\psi_{\epsilon}(s)\|_{H}^{2}\;ds
+\int_{0}^{t} K(s)(1+\|\tilde{X}^{\epsilon}_{s}\|_{H}^{2})\;ds\nonumber\\
&\leq& \int_{0}^{t}\|
\tilde{X}^{\epsilon}_{s}\|^{2}_{H}\left(\|\psi_{\epsilon}(s)\|_{H}^{2}+K(s)\right)\;ds+\int_{0}^{t}K(s)\;ds.
\end{eqnarray}

Denote the forth term on the right hand side of equation $(\ref{Ito formula})$ by $W_{t}$. Then we have
\begin{eqnarray}
\label{estimates-martingale 1}
E[\sup_{0\leq s \leq T}|W_{s}|]&\leq& 4E\sqrt{[W]_{T}}\nonumber\\
&\leq& 8\sqrt{\epsilon}
E[\sqrt{\int_{0}^{T}\|\sigma(s,\tilde{X}^{\epsilon}_{s})\|_{L_{2}(H)}^{2}\|\tilde{X}^{\epsilon}_{s}\|_{H}^{2}\;ds}]\nonumber\\
&\leq& 8\sqrt{\epsilon} E[\sup_{0\leq s \leq T}\|\tilde{X}^{\epsilon}_{s}\|_{H}\cdot
\sqrt{\int_{0}^{T}\|\sigma(s,\tilde{X}^{\epsilon}_{s})\|_{L_{2}(H)}^{2}\;ds}]\nonumber\\
&\leq& 4\sqrt{\epsilon} E[\sup_{0\leq s \leq T}\|\tilde{X}^{\epsilon}_{s}\|_{H}^{2}]+4\sqrt{\epsilon}
E[\int_{0}^{T}\|\sigma(s,\tilde{X}^{\epsilon}_{s})\|_{L_{2}(H)}^{2}\;ds]\nonumber\\
&\leq&  4\sqrt{\epsilon} E[\sup_{0\leq s \leq T}\|\tilde{X}^{\epsilon}_{s}\|_{H}^{2}]+4\sqrt{\epsilon}
E[\int_{0}^{T}K(s)(1+\|\tilde{X}^{\epsilon}_{s}\|_{H}^{2})\;ds]\nonumber\\
&\leq&  4\sqrt{\epsilon}(1+\int_{0}^{T}K(s)\;ds)E[\sup_{0\leq s \leq
T}\|\tilde{X}^{\epsilon}_{s}\|_{H}^{2}]+4\sqrt{\epsilon} E[\int_{0}^{T}K(s)\;ds].
\end{eqnarray}

The fifth term on the right hand side of equation $(\ref{Ito formula})$ has the following bound.
\begin{eqnarray}
\label{estimates-fifth term}
&&2\int_{0}^{t}\int_{\mathbb{X}}<\tilde{X}^{\epsilon}_{s},G(s,\tilde{X}^{\epsilon}_{s},v)>(\varphi_{\epsilon}(s,v)-1)\nu(dv)ds\nonumber\\
&\leq&
2\int_{0}^{t}\int_{\mathbb{X}}\|\tilde{X}^{\epsilon}_{s}\|_{H}\|G(s,\tilde{X}^{\epsilon}_{s},v)\|_{H}|\varphi_{\epsilon}(s,v)-1|\nu(dv)ds\nonumber\\
&\leq&
2\int_{0}^{t}\int_{\mathbb{X}}\|G(s,v)\|_{0,H}(1+\|\tilde{X}^{\epsilon}_{s}\|_{H})\|\tilde{X}^{\epsilon}_{s}\|_{H}|\varphi_{\epsilon}(s,v)-1|\nu(dv)ds\nonumber\\
&\leq&
2\int_{0}^{t}\int_{\mathbb{X}}\|G(s,v)\|_{0,H}(1+2\|\tilde{X}^{\epsilon}_{s}\|^{2}_{H})|\varphi_{\epsilon}(s,v)-1|\nu(dv)ds\nonumber\\
&\leq&
2C^{N}_{0,1}+4\int_{0}^{t}\int_{\mathbb{X}}\|G(s,v)\|_{0,H}\|\tilde{X}^{\epsilon}_{s}\|^{2}_{H}|\varphi_{\epsilon}(s,v)-1|\nu(dv)ds,
\end{eqnarray}
where  $C^{N}_{0,1}$ is the constant from Lemma $\ref{Lemma-Condition-0,H-1,H}$.

The seventh term in $(\ref{Ito formula})$ is bounded by,
\begin{eqnarray}
\label{estimates-seventh term}
&&\epsilon\int_{0}^{t}\int_{\mathbb{X}}\|G(s,\tilde{X}^{\epsilon}_{s},v)\|_{H}^{2}\varphi_{\epsilon}(s,v)\nu(dv)ds\nonumber\\
&&\leq 2\epsilon \int_{0}^{t}\int_{\mathbb{X}}\|
G(s,v)\|_{0,H}^{2}(1+\|\tilde{X}^{\epsilon}_{s}\|_{H}^{2})\varphi_{\epsilon}(s,v)\nu(dv)ds\nonumber\\
&&\leq 2\epsilon C_{0,2}^{N}+ 2\epsilon \int_{0}^{t}\int_{\mathbb{X}}\|
G(s,v)\|_{0,H}^{2}\|\tilde{X}^{\epsilon}_{s}\|_{H}^{2}\varphi_{\epsilon}(s,v)\nu(dv)ds,
\end{eqnarray}
where constant $C^{N}_{0,2}$ was defined in Lemma $\ref{Lemma-Condition-0,H-1,H}$.

The last term in $(\ref{Ito formula})$ is bounded by,
\begin{eqnarray}
\label{estimates-last term}
\epsilon\int_{0}^{t}\|\sigma(s,\tilde{X}^{\epsilon}_{s})\|_{L_{2}(H)}^{2}\;ds
&\leq& \epsilon\int_{0}^{t}K(s)(1+\|\tilde{X}^{\epsilon}_{s}\|^{2}_{H})\;ds\nonumber\\
&\leq&  \epsilon\int_{0}^{t}\|\tilde{X}^{\epsilon}_{s}\|^{2}_{H}K(s)\;ds+\epsilon\int_{0}^{t}K(s)\;ds.
\end{eqnarray}

Set
$$
M_{t}=\int_{0}^{t}\int_{\mathbb{X}}<2\epsilon G(s,\tilde{X}^{\epsilon}_{s-},v),\tilde{X}^{\epsilon}_{s}>\left(
N^{\epsilon^{-1}\varphi_{\epsilon}}(ds dv)-\epsilon^{-1}\varphi_{\epsilon}(s,v)\nu(dv)ds\right)
$$
and
$$
K_{t}=\int_{0}^{t}\int_{\mathbb{X}}\|\epsilon G(s,\tilde{X}^{\epsilon}_{s-},v)\|_{H}^{2}\left(
N^{\epsilon^{-1}\varphi_{\epsilon}}(ds dv)-\epsilon^{-1}\varphi_{\epsilon}(s,v)\nu(dv)ds\right).
$$

We have
\begin{eqnarray}
\label{estimates-martingale 2}
\mathbb{E}[\sup_{0\leq t \leq T}|K_{t}|]&\leq&\mathbb{E}[\sup_{0\leq t \leq T}\int_{0}^{t}\int_{\mathbb{X}}\|\epsilon
G(s,\tilde{X}^{\epsilon}_{s-},v)\|_{H}^{2}N^{\epsilon^{-1}\varphi_{\epsilon}}(ds dv)]
\nonumber\\
&&+\epsilon \mathbb{E}[\int_{0}^{T}\int_{\mathbb{X}}\|
G(s,\tilde{X}^{\epsilon}_{s},v)\|_{H}^{2}\varphi_{\epsilon}(s,v)\nu(dv)ds]\nonumber\\
&&\leq 2\epsilon \mathbb{E}[\int_{0}^{T}\int_{\mathbb{X}}\|
G(s,\tilde{X}^{\epsilon}_{s},v)\|_{H}^{2}\varphi_{\epsilon}(s,v)\nu(dv)ds]\nonumber\\
&&\leq 4\epsilon \mathbb{E}[\int_{0}^{T}\int_{\mathbb{X}}\|
G(s,v)\|_{0,H}^{2}(1+\|\tilde{X}^{\epsilon}_{s}\|_{H}^{2})\varphi_{\epsilon}(s,v)\nu(dv)ds]\nonumber\\
&&\leq 4\epsilon C_{0,2}^{N}(1+\mathbb{E}[\sup_{0\leq t \leq T}\|\tilde{X}^{\epsilon}_{s}\|_{H}^{2}]).
\end{eqnarray}

For the martingale $M_{t}$, we have
\begin{eqnarray}
\label{estimates-martingale 3}
\mathbb{E}[\sup_{0\leq t \leq T}|M_{t}|]&\leq& 4\mathbb{E}\sqrt{[M]_{T}}\nonumber\\
&&\leq 4\mathbb{E}\left\{\int_{0}^{T}\int_{\mathbb{X}}|<2\epsilon
G(s,\tilde{X}^{\epsilon}_{s-},v),\tilde{X}^{\epsilon}_{s}>|^{2} N^{\epsilon^{-1}\varphi_{\epsilon}}(ds
dv)\right\}^{\frac{1}{2}}\nonumber\\
&&\leq 8\epsilon \mathbb{E}\left\{\int_{0}^{T}\int_{\mathbb{X}}
\|G(s,\tilde{X}^{\epsilon}_{s-},v)\|_{H}^{2}\|\tilde{X}^{\epsilon}_{s}\|_{H}^{2}
N^{\epsilon^{-1}\varphi_{\epsilon}}(ds dv)\right\}^{\frac{1}{2}}\nonumber\\
&&\leq 8\epsilon \mathbb{E}(\sup_{0\leq t \leq
T}\|\tilde{X}^{\epsilon}_{s}\|_{H})\left\{\int_{0}^{T}\int_{\mathbb{X}}
\|G(s,\tilde{X}^{\epsilon}_{s-},v)\|_{H}^{2} N^{\epsilon^{-1}\varphi_{\epsilon}}(ds
dv)\right\}^{\frac{1}{2}}\nonumber\\
&&\leq \frac{1}{C}\mathbb{E}[\sup_{0\leq t \leq
T}\|\tilde{X}^{\epsilon}_{s}\|^{2}_{H}]+16\epsilon^{2}C\mathbb{E}[\int_{0}^{T}\int_{\mathbb{X}}
\|G(s,\tilde{X}^{\epsilon}_{s-},v)\|_{H}^{2} N^{\epsilon^{-1}\varphi_{\epsilon}}(ds dv)]\nonumber\\
&&=\frac{1}{C}\mathbb{E}[\sup_{0\leq t \leq T}\|\tilde{X}^{\epsilon}_{s}\|^{2}_{H}]+16\epsilon
C\mathbb{E}[\int_{0}^{T}\int_{\mathbb{X}} \|G(s,\tilde{X}^{\epsilon}_{s-},v)\|_{H}^{2}
\varphi_{\epsilon}(s,v)\nu(dv)ds]\nonumber\\
&&\leq \frac{1}{C}\mathbb{E}[\sup_{0\leq t \leq T}\|\tilde{X}^{\epsilon}_{s}\|^{2}_{H}]+32\epsilon C
C_{0,2}^{N}(1+\mathbb{E}[\sup_{0\leq t \leq T}\|\tilde{X}^{\epsilon}_{s}\|_{H}^{2}])\nonumber\\
&&\leq 32\epsilon C C_{0,2}^{N}+(\frac{1}{C}+32\epsilon C C_{0,2}^{N})\mathbb{E}[\sup_{0\leq t \leq
T}\|\tilde{X}^{\epsilon}_{s}\|_{H}^{2}],
\end{eqnarray}
where $C$ is any positive number.

Set $K=\int_{0}^{T}K(s)ds$. Combining the estimates $(\ref{estimates-third term})$, $(\ref{estimates-fifth term})$,
$(\ref{estimates-seventh term})$ and $(\ref{estimates-last term})$, we have
\begin{eqnarray}
&&\|\tilde{X}^{\epsilon}_{t}\|_{H}^{2}\nonumber\\
&\leq& (\|\tilde{X}_{0}\|_{H}^{2}+(1+\epsilon)K+\sup_{0\leq t \leq T}|W_{t}|+
2C^{N}_{0,1}+ 2\epsilon C_{0,2}^{N}+\sup_{0\leq t \leq T}|K_{t}|+\sup_{0\leq t \leq T}|M_{t}|)\nonumber\\
&&+\int_{0}^{t}\|\tilde{X}^{\epsilon}_{s}\|^{2}_{H} \cdot(\|\psi_{\epsilon}(s)\|_{H}^{2}+(1+\epsilon)K(s)
+\int_{\mathbb{X}}(4\|G(s,v)\|_{0,H}|\varphi_{\epsilon}(s,v)-1| \nonumber\\
&&\quad +2\epsilon \| G(s,v)\|_{0,H}^{2}\varphi_{\epsilon}(s,v))\nu(dv))\;ds.\nonumber
\end{eqnarray}

By Grownwall's inequality and Lemma $\ref{Lemma-Condition-0,H-1,H}$, we get
\begin{eqnarray}
\|\tilde{X}^{\epsilon}_{t}\|_{H}^{2}
&\leq& (\|\tilde{X}_{0}\|_{H}^{2}+(1+\epsilon)K+\sup_{0\leq t \leq T}|W_{t}|+
2C^{N}_{0,1}+ 2\epsilon C_{0,2}^{N}+\sup_{0\leq t \leq T}|K_{t}|+\sup_{0\leq t \leq T}|M_{t}|)\nonumber\\
&&\cdot e^{(N+(1+\epsilon)K+4C^{N}_{0,1}+2\epsilon C_{0,2}^{N})}.\nonumber
\end{eqnarray}

Set $C_{0}=e^{(N+(1+\epsilon)K+4C^{N}_{0,1}+2\epsilon C_{0,2}^{N})}$. By $(\ref{estimates-martingale 1})$,
$(\ref{estimates-martingale 2})$ and
$(\ref{estimates-martingale 3})$, we have
\begin{eqnarray}
\mathbb{E}[\sup_{0\leq t \leq T}\|\tilde{X}^{\epsilon}_{t}\|_{H}^{2}]&\leq&
C_{0}\left(\mathbb{E}\|\tilde{X}_{0}\|_{H}^{2}+2C^{N}_{0,1}+ 2\epsilon C_{0,2}^{N}+4\epsilon C_{0,2}^{N}+32\epsilon C
C_{0,2}^{N}+(1+\epsilon+4\sqrt{\epsilon})K\right)\nonumber\\
&&+C_{0}\left(4\epsilon C_{0,2}^{N}+\frac{1}{C}+32\epsilon C
C_{0,2}^{N}+4\sqrt{\epsilon}(1+K)\right)\mathbb{E}[\sup_{0\leq t \leq
T}\|\tilde{X}^{\epsilon}_{s}\|_{H}^{2}].\nonumber
\end{eqnarray}

Since constant $C$ can be arbitrarily large, we can select $C$ and  $\epsilon_{0}$ small enough,
such that
$$
C_{0}\left(4\epsilon_{0} C_{0,2}^{N}+\frac{1}{C}+32\epsilon_{0} C
C_{0,2}^{N}+4\sqrt{\epsilon_{0}}(1+K)\right)<\frac{1}{2}
$$

Therefore, we have
\begin{eqnarray*}
&&\sup_{0<\epsilon<\epsilon_{0}}\mathbb{E}\sup_{0\leq t\leq T}\|\tilde{X}^{\epsilon}_{t}\|_{H}^{2}\nonumber\\
&\leq&
2C_{0}\left(\mathbb{E}\|\tilde{X}_{0}\|_{H}^{2}+2C^{N}_{0,1}+ 2\epsilon_{0} C_{0,2}^{N}+4\epsilon_{0}
C_{0,2}^{N}+32\epsilon_{0} C C_{0,2}^{N}+(1+\epsilon_{0}+4\sqrt{\epsilon_{0}})K\right)\nonumber\\
&<&\infty.
\end{eqnarray*}
\end{proof}

The following proof of estimates (\ref{estimation-02}) and (\ref{estimation-03}) was inspired by the method in \cite{Rockner-Zhang}.
\begin{lem}
There exists $\epsilon_{0}>0$, such that, for any $t_0\in[0,T]$, $0<\epsilon<\epsilon_0$,
\begin{eqnarray}\label{estimation-02}
     &  &(\frac{1}{2}-4\epsilon^{1/2})\mathbb{E}[\sup_{0\leq t\leq t_0}\sum_{i=k}^\infty|\langle
     \tilde{X}^{\epsilon}_{t},e_i\rangle|^2]\\
&\leq&
        \sum_{i=k}^\infty\langle \tilde{X}^{\epsilon}_{0},e_i\rangle^2
     +
            (8N+129\epsilon)(\int_{0}^{t_{0}}K(s)ds)\mathbb{E}[(1+\sup_{0\leq t\leq
            t_0}\|\tilde{X}^{\epsilon}_{t}\|_H)^2]\nonumber\\
    && +4\mathbb{E}[(1+\sup_{0\leq t\leq t_0}\|\tilde{X}^{\epsilon}_{t}\|_H)^2]
        \cdot
            \sup_{g\in
            S^N}\left(\int_0^{t_0}\int_{\mathbb{X}}\|G(s,v)\|_{0,H}|g(s,v)-1|\nu(dv)ds\right)^{2}\nonumber\\
     &  &+
             (4\epsilon^{1/2}+2\epsilon)\mathbb{E}[(1+\sup_{0\leq t\leq t_0}\|\tilde{X}^{\epsilon}_{t}\|_H)^2]
        \cdot
             \sup_{g\in S^N}\int_0^{t_0}\int_{\mathbb{X}}\|G(s,v)\|^2_{0,H}g(s,v)\nu(dv)ds.\nonumber
\end{eqnarray}
\end{lem}
\begin{proof}
Recall that  $\{e_k,\ k\geq 1\}$ is an orthonormal basis of the Hilbert space $H$. For convenience, let
${Y}_i^\epsilon=\langle \widetilde{X}^\epsilon,e_i\rangle$. Then
\begin{eqnarray*}
Y^{\epsilon}_{i}(t)&=&<\tilde{X}^{\epsilon}_{t},e_{i}>\nonumber\\
&=&<X_{0},e_{i}>-\zeta_{i}\int_{0}^{t}Y^{\epsilon}_{i}(s)ds
+\int_{0}^{t}\langle\sigma(s,\tilde{X}^{\epsilon}_{s})\psi_{\epsilon}(s),e_{i}\rangle \;ds
+\sqrt{\epsilon}\langle\int_{0}^{t}\sigma(s,\tilde{X}^{\epsilon}_{s})d\beta(s),e_{i}\rangle\nonumber\\
&&+\int_{0}^{t}\int_{\mathbb{X}}<G(s,\tilde{X}^{\epsilon}_{s},v),e_{i}>
(\varphi_{\epsilon}(s,v)-1)\nu(dv)ds\nonumber\\
&&+\int_{0}^{t}\int_{\mathbb{X}} <\epsilon G(s,\tilde{X}^{\epsilon}_{s-},v),e_{i}>
\tilde{N}^{\epsilon^{-1}\varphi_{\epsilon}}(ds dv).
\end{eqnarray*}

By Ito's formula,
\begin{eqnarray}
\label{ITO formula 2}
|Y^{\epsilon}_{i}(t)|^{2}&=&|<X_{0},e_{i}>|^{2}-2\zeta_{i}\int_{0}^{t}|Y^{\epsilon}_{i}|^{2}(s)ds
+2\int_{0}^{t}Y^{\epsilon}_{i}(s)\langle\sigma(s,\tilde{X}^{\epsilon}_{s})\psi_{\epsilon}(s),e_{i}\rangle \;ds
\nonumber\\
&&+2\sqrt{\epsilon}\langle \int_{0}^{t}Y^{\epsilon}_{i}(s)\sigma(s,\tilde{X}^{\epsilon}_{s})d\beta(s),e_{i}\rangle
+\epsilon\int_{0}^{t}\|\sigma^{*}(s,\tilde{X}^{\epsilon}_{s})e_{i}\|_{H}^{2}\;ds\nonumber\\
&&+2\int_{0}^{t}\int_{\mathbb{X}}Y^{\epsilon}_{i}(s)<G(s,\tilde{X}^{\epsilon}_{s},v),e_{i}>
(\varphi_{\epsilon}(s,v)-1)\nu(dv)ds\nonumber\\
&&+\int_{0}^{t}\int_{\mathbb{X}}\left(<2\epsilon G(s,\tilde{X}^{\epsilon}_{s-},v),Y^{\epsilon}_{i}(s-)e_{i}>
+|\epsilon <G(s,\tilde{X}^{\epsilon}_{s-},v),e_{i}>|^{2}\right)\tilde{N}^{\epsilon^{-1}\varphi_{\epsilon}}(ds
dv)\nonumber\\
&&+\epsilon
\int_{0}^{t}\int_{\mathbb{X}}|<G(s,\tilde{X}^{\epsilon}_{s-},v),e_{i}>|^{2}\varphi_{\epsilon}(s,v)\nu(dv)ds.
\end{eqnarray}

Therefore, for $t_{0}>0$, we obtain
\begin{eqnarray}
\label{suprem of Yi}
&&\sup_{0\leq t\leq t_{0}}\sum_{i=k}^{\infty}|Y^{\epsilon}_{i}(t)|^{2}\nonumber\\
&\leq& \sum_{i=k}^{\infty}|<X_{0},e_{i}>|^{2}
+2\int_{0}^{t_{0}}\|\sigma(s,\tilde{X}^{\epsilon}_{s})\|_{L_{2}(H)}\|\psi_{\epsilon}(s)\|_{H}
\|\sum_{i=k}^{\infty}Y^{\epsilon}_{i}(s)e_{i}\|_{H}\;ds\nonumber\\
&&+2\sqrt{\epsilon}\sup_{0\leq t\leq t_{0}}\sum_{i=k}^{\infty}\langle
\int_{0}^{t}Y^{\epsilon}_{i}(s)\sigma(s,\tilde{X}^{\epsilon}_{s})d\beta(s),e_{i}\rangle
+\epsilon\int_{0}^{t_{0}}\sum_{i=k}^{\infty}\|\sigma^{*}(s,\tilde{X}^{\epsilon}_{s})e_{i}\|_{H}^{2}\;ds\nonumber\\
&&+2\int_{0}^{t_{0}}\int_{\mathbb{X}}|<G(s,\tilde{X}^{\epsilon}_{s},v),\sum_{i=k}^{\infty}Y^{\epsilon}_{i}(s)e_{i}>
(\varphi_{\epsilon}(s,v)-1)|\nu(dv)ds\nonumber\\
&&+2\sup_{0\leq t\leq t_{0}}\int_{0}^{t}\int_{\mathbb{X}}|<\epsilon G(s,\tilde{X}^{\epsilon}_{s-},v),
\sum_{i=k}^{\infty}Y^{\epsilon}_{i}(s-)e_{i}> |\tilde{N}^{\epsilon^{-1}\varphi_{\epsilon}}(ds dv)\nonumber\\
&&+2\int_{0}^{t_{0}}\int_{\mathbb{X}}\|\epsilon
G(s,\tilde{X}^{\epsilon}_{s},v)\|_{H}^{2}N^{\epsilon^{-1}\varphi_{\epsilon}}(ds dv).
\end{eqnarray}

Firstly, we estimate the second term on the right hand side of the above inequality.
\begin{eqnarray}
&&2E[\int_{0}^{t_{0}}\|\sigma(s,\tilde{X}^{\epsilon}_{s})\|_{L_{2}(H)}\|\psi_{\epsilon}(s)\|_{H}
\|\sum_{i=k}^{\infty}Y^{\epsilon}_{i}(s)e_{i}\|_{H}\;ds]\nonumber\\
&\leq&  \frac{1}{8N}E[\int_{0}^{t_{0}}
\|\psi_{\epsilon}(s)\|_{H}^{2}\|\sum_{i=k}^{\infty}Y^{\epsilon}_{i}(s)e_{i}\|_{H}^{2}\;ds]
+8N E[\int_{0}^{t_{0}}K(s)(1+\|\tilde{X}^{\epsilon}_{s}\|_{H}^{2})\;ds]\nonumber\\
&\leq& \frac{1}{8} E[\sup_{0\leq t\leq t_{0}}\sum_{i=k}^{\infty}|Y^{\epsilon}_{i}(t)|^{2}]+8N(\int_{0}^{t_{0}}K(s)ds)
E[(1+\sup_{0\leq t\leq t_{0}}\|\tilde{X}^{\epsilon}_{t}\|_{H})^{2}].
\end{eqnarray}

The third term is bounded by,
\begin{eqnarray}
&&2\sqrt{\epsilon}E[\sup_{0\leq t\leq t_{0}}\sum_{i=k}^{\infty}\langle
\int_{0}^{t}Y^{\epsilon}_{i}(s)\sigma(s,\tilde{X}^{\epsilon}_{s})d\beta(s),e_{i}\rangle]\nonumber\\
&\leq& 8\sqrt{\epsilon}
E[\sqrt{\int_{0}^{t_{0}}\|\sigma(s,\tilde{X}^{\epsilon}_{s})\|^{2}_{L_{2}(H)}\|\sum_{i=k}^{\infty}|Y^{\epsilon}_{i}(t)e_{i}\|^{2}_{H}\;ds}]\nonumber\\
&\leq& 8\sqrt{\epsilon} E[(\sup_{0\leq t\leq
t_{0}}\|\sum_{i=k}^{\infty}Y^{\epsilon}_{i}(t)e_{i}\|_{H})\sqrt{\int_{0}^{t_{0}}\|\sigma(s,\tilde{X}^{\epsilon}_{s})\|^{2}_{L_{2}(H)}\;ds}]\nonumber\\
&\leq& \frac{1}{8} E[\sup_{0\leq t\leq t_{0}}\sum_{i=k}^{\infty}|Y^{\epsilon}_{i}(t)|^{2}]+128\epsilon
E[\int_{0}^{t_{0}}\|\sigma(s,\tilde{X}^{\epsilon}_{s})\|^{2}_{L_{2}(H)}\;ds]\nonumber\\
&\leq& \frac{1}{8} E[\sup_{0\leq t\leq t_{0}}\sum_{i=k}^{\infty}|Y^{\epsilon}_{i}(t)|^{2}]+128\epsilon
(\int_{0}^{t_{0}}K(s)ds) E[(1+\sup_{0\leq t\leq t_{0}}\|\tilde{X}^{\epsilon}_{t}\|_{H})^{2}].
\end{eqnarray}

The forth term is estimated as follows,
\begin{eqnarray}
&&\epsilon
E[\int_{0}^{t_{0}}\sum_{i=k}^{\infty}\|\sigma^{*}(s,\tilde{X}^{\epsilon}_{s})e_{i}\|_{H}^{2}\;ds]\nonumber\\
&\leq& \epsilon E[\int_{0}^{t_{0}}\|\sigma(s,\tilde{X}^{\epsilon}_{s})\|_{L_{2}(H)}^{2}\;ds]\nonumber\\
&\leq&  \epsilon(\int_{0}^{t_{0}}K(s)ds) E[(1+\sup_{0\leq t\leq t_{0}}\|\tilde{X}^{\epsilon}_{t}\|_{H})^{2}].
\end{eqnarray}

The fifth term is bounded by,
\begin{eqnarray}
&&E[\int_{0}^{t_{0}}\int_{\mathbb{X}}|<G(s,\tilde{X}^{\epsilon}_{s},v),\sum_{i=k}^{\infty}Y^{\epsilon}_{i}(s)e_{i}>|\cdot
|\varphi_{\epsilon}(s,v)-1|\nu(dv)ds]\nonumber\\
&\leq&
E[\int_{0}^{t_{0}}\int_{\mathbb{X}}\|G(s,\tilde{X}^{\epsilon}_{s},v)\|_{H}(\sum_{i=k}^{\infty}|Y^{\epsilon}_{i}(s)|^{2})^{\frac{1}{2}}
|\varphi_{\epsilon}(s,v)-1|\nu(dv)ds]\nonumber\\
&\leq& E[\sup_{0\leq t\leq t_{0}}(\sum_{i=k}^{\infty}|Y^{\epsilon}_{i}(s)|^{2})^{\frac{1}{2}}\cdot
\int_{0}^{t_{0}}\int_{\mathbb{X}}\|G(s,\tilde{X}^{\epsilon}_{s},v)\|_{H}|\varphi_{\epsilon}(s,v)-1|\nu(dv)ds]\nonumber\\
&\leq& \frac{1}{8}E[\sup_{0\leq t\leq
t_{0}}(\sum_{i=k}^{\infty}|Y^{\epsilon}_{i}(s)|^{2})]+2E[(\int_{0}^{t_{0}}\int_{\mathbb{X}}\|G(s,\tilde{X}^{\epsilon}_{s},v)\|_{H}|\varphi_{\epsilon}(s,v)-1|\nu(dv)ds)^{2}]\nonumber\\
&\leq&\frac{1}{8}E[\sup_{0\leq t\leq t_{0}}(\sum_{i=k}^{\infty}|Y^{\epsilon}_{i}(s)|^{2})]\nonumber\\
&+&2E[(1+\sup_{0\leq t\leq t_{0}}\|\tilde{X}^{\epsilon}_{t}\|)^{2}]\cdot
\sup_{g\in \mathbb{S}^{N}}(\int_{0}^{t_{0}}\int_{\mathbb{X}}\|G(s,v)\|_{0,H}|g(s,v)-1|\nu(dv)ds)^{2}
\end{eqnarray}

Set
$$
M_{t}:=\int_{0}^{t}\int_{\mathbb{X}}<\epsilon G(s,\tilde{X}^{\epsilon}_{s-},v),
\sum_{i=k}^{\infty}Y^{\epsilon}_{i}(s-)e_{i}> \tilde{N}^{\epsilon^{-1}\varphi_{\epsilon}}(ds dv).
$$

Then we have,
\begin{eqnarray}
E[\sup_{0\leq t\leq t_{0}}|M_{t}|]
&\leq& 4E\sqrt{[M]_{t_{0}}} \nonumber\\
&\leq & 4E\{\int_{0}^{t_{0}}\int_{\mathbb{X}}|<\epsilon G(s,\tilde{X}^{\epsilon}_{s-},v),
\sum_{i=k}^{\infty}Y^{\epsilon}_{i}(s-)e_{i}>|^{2} N^{\epsilon^{-1}\varphi_{\epsilon}}(ds
dv)\}^{\frac{1}{2}}\nonumber\\
&\leq& 4E\{\int_{0}^{t_{0}}\int_{\mathbb{X}}\|\epsilon G(s,\tilde{X}^{\epsilon}_{s-},v)\|^{2}
(\sum_{i=k}^{\infty}|Y^{\epsilon}_{i}(s-)|^{2} ) N^{\epsilon^{-1}\varphi_{\epsilon}}(ds
dv)\}^{\frac{1}{2}}\nonumber\\
&\leq & 4E[\{\epsilon^{\frac{1}{2}}\sup_{0\leq t\leq t_{0}}\sum_{i=k}^{\infty}|Y^{\epsilon}_{i}(t)|^{2}
\}^{\frac{1}{2}}\cdot
\{\int_{0}^{t_{0}}\int_{\mathbb{X}}\epsilon^{\frac{3}{2}}\| G(s,\tilde{X}^{\epsilon}_{s-},v)\|^{2}
N^{\epsilon^{-1}\varphi_{\epsilon}}(ds dv)\}^{\frac{1}{2}}]\nonumber\\
&\leq & 2E[\epsilon^{\frac{1}{2}}\sup_{0\leq t\leq t_{0}}\sum_{i=k}^{\infty}|Y^{\epsilon}_{i}(t)|^{2}]+
2E[\int_{0}^{t_{0}}\int_{\mathbb{X}}\epsilon^{\frac{3}{2}}\| G(s,\tilde{X}^{\epsilon}_{s-},v)\|^{2}
N^{\epsilon^{-1}\varphi_{\epsilon}}(ds dv)]\nonumber\\
&\leq & 2\epsilon^{\frac{1}{2}}E[\sup_{0\leq t\leq t_{0}}\sum_{i=k}^{\infty}|Y^{\epsilon}_{i}(t)|^{2}]+
2\epsilon^{\frac{1}{2}}E[\int_{0}^{t_{0}}\int_{\mathbb{X}}\| G(s,\tilde{X}^{\epsilon}_{s-},v)\|^{2}
\varphi_{\epsilon}(s,v)\nu(v)ds]\nonumber\\
&\leq & 2\epsilon^{\frac{1}{2}}E[\sup_{0\leq t\leq t_{0}}\sum_{i=k}^{\infty}|Y^{\epsilon}_{i}(t)|^{2}]\nonumber\\
&&+2\epsilon^{\frac{1}{2}}E[(1+\sup_{0\leq t\leq
t_{0}}\|\tilde{X}^{\epsilon}_{t}\|)^{2}](\int_{0}^{t_{0}}\int_{\mathbb{X}}\| G(s,v)\|^{2}_{0,H}
\varphi_{\epsilon}(s,v)\nu(v)ds)\nonumber\\
&\leq & 2\epsilon^{\frac{1}{2}}E[\sup_{0\leq t\leq t_{0}}\sum_{i=k}^{\infty}|Y^{\epsilon}_{i}(t)|^{2}]\nonumber\\
&&+
2\epsilon^{\frac{1}{2}}E[(1+\sup_{0\leq t\leq t_{0}}\|\tilde{X}^{\epsilon}_{t}\|)^{2}]\cdot\sup_{g\in
\mathbb{S}^{N}}(\int_{0}^{t_{0}}\int_{\mathbb{X}}\| G(s,v)\|^{2}_{0,H} g(s,v)\nu(v)ds)\nonumber\\
\end{eqnarray}

The upper bound of the last term in $(\ref{suprem of Yi})$ is given by,
\begin{eqnarray}
&&E[\int_{0}^{t_{0}}\int_{\mathbb{X}}\|\epsilon
G(s,\tilde{X}^{\epsilon}_{s},v)\|_{H}^{2}N^{\epsilon^{-1}\varphi_{\epsilon}}(ds dv)]\nonumber\\
&\leq& \epsilon E[ \int_{0}^{t_{0}}\int_{\mathbb{X}}\|
G(s,\tilde{X}^{\epsilon}_{s},v)\|_{H}^{2}\varphi_{\epsilon}(s,v)\nu(dv)ds]\nonumber\\
&\leq& \epsilon E[(1+\sup_{0\leq t\leq t_{0}}\|\tilde{X}^{\epsilon}_{t}\|)^{2}]\cdot\sup_{g\in
\mathbb{S}^{N}}(\int_{0}^{t_{0}}\int_{\mathbb{X}}\| G(s,v)\|^{2}_{0,H} g(s,v)\nu(v)ds).
\end{eqnarray}

Therefore, combining the above inequalities, we get
\begin{eqnarray}
&&E[\sup_{0\leq t\leq t_{0}}\sum_{i=k}^{\infty}|Y^{\epsilon}_{i}(t)|^{2}]\nonumber\\
&\leq& \sum_{i=k}^{\infty}|<X_{0},e_{i}>|^{2}+(\frac{1}{2}+4\epsilon^{\frac{1}{2}})E[\sup_{0\leq t\leq
t_{0}}(\sum_{i=k}^{\infty}|Y^{\epsilon}_{i}(s)|^{2})]\nonumber\\
&&+(8N+129\epsilon)(\int_{0}^{t_{0}}K(s)ds)E[(1+\sup_{0\leq t\leq t_{0}}\|\tilde{X}^{\epsilon}_{t}\|)^{2}]\nonumber\\
&&+4E[(1+\sup_{0\leq t\leq t_{0}}\|\tilde{X}^{\epsilon}_{t}\|)^{2}]\cdot \sup_{g\in
\mathbb{S}^{N}}(\int_{0}^{t_{0}}\int_{\mathbb{X}}\|G(s,v)\|_{0,H}|g(s,v)-1|\nu(dv)ds)^{2}\nonumber\\
&&+(4\epsilon^{\frac{1}{2}}+2\epsilon)E[(1+\sup_{0\leq t\leq t_{0}}\|\tilde{X}^{\epsilon}_{t}\|)^{2}]\cdot \sup_{g\in
\mathbb{S}^{N}}(\int_{0}^{t_{0}}\int_{\mathbb{X}}\| G(s,v)\|^{2}_{0,H} g(s,v)\nu(v)ds)
\end{eqnarray}
\end{proof}
\begin{lem}
There exists $\epsilon_{0}>0$, such that, for any $t_0\in[0,T]$,
\begin{eqnarray}\label{estimation-03}
   \sup_{0<\epsilon<\epsilon_0}\sup_{\phi_\epsilon=(\psi_\epsilon,\varphi_\epsilon)\in\widetilde{\mathcal{U}}^N}\mathbb{E}\sup_{t_0\leq t\leq T}\sum_{i=k}^\infty|\langle
   \tilde{X}^{\epsilon}_{t},e_i\rangle|^2
\leq
    e^{-2\zeta_kt_0}C.
\end{eqnarray}
\end{lem}
\begin{proof}
By (\ref{ITO formula 2}), we have
\begin{eqnarray*}
      \sum_{i=k}^\infty|Y_i^\epsilon(t)|^2
&\leq&
        \sum_{i=k}^\infty\langle X_0,e_i\rangle^2-2\zeta_k\int_0^t\sum_{i=k}^\infty|Y_i^\epsilon(s)|^2ds\\
    &  &+
        2\int_0^t\int_\mathbb{X}\langle G(s,\widetilde{X}_s^\epsilon,v),\sum_{i=k}^\infty
        Y_i^\epsilon(s)e_i)\rangle(\varphi_\epsilon(s,v)-1)\nu(dv)ds\\
    &  &+
        2\int_0^t\int_\mathbb{X}\langle \epsilon G(s,\widetilde{X}_s^\epsilon,v),\sum_{i=k}^\infty
        Y_i^\epsilon(s)e_i)\rangle \widetilde{N}^{\epsilon^{-1}\varphi_\epsilon}(dv,ds)\\
    &  &+
         \int_0^t\int_\mathbb{X}\sum_{i=k}^\infty|\langle \epsilon G(s,\widetilde{X}_s^\epsilon,v),e_i)\rangle|^2
         N^{\epsilon^{-1}\varphi_\epsilon}(dv,ds)\nonumber\\
     &  &+
         2\int_0^t\langle\sqrt{\epsilon}\sigma(s,\widetilde{X}_s^\epsilon)d\beta_s,\sum_{i=k}^\infty
         Y_i^\epsilon(s)e_i\rangle\\
     &   &+
          2\int_0^t\langle\sigma(s,\widetilde{X}_s^\epsilon)\psi_\epsilon(s),\sum_{i=k}^\infty
          Y_i^\epsilon(s)e_i\rangle ds\\
     &   &+
          \epsilon\int_0^t\|\sigma(s,\widetilde{X}_s^\epsilon)\|^2_{L_2(H)}ds\\
&\leq&
        \sum_{i=k}^\infty\langle X_0,e_i\rangle^2-2\zeta_k\int_0^t\sum_{i=k}^\infty|Y_i^\epsilon(s)|^2ds\\
    &  &+
        2\int_0^T\int_\mathbb{X}|\langle G(s,\widetilde{X}_s^\epsilon,v),\sum_{i=k}^\infty
        Y_i^\epsilon(s)e_i)\rangle(\varphi_\epsilon(s,v)-1)|\nu(dv)ds\\
    &  &+
        2\sup_{0\leq t'\leq T}|\int_0^{t'}\int_\mathbb{X}\langle \epsilon
        G(s,\widetilde{X}_s^\epsilon,v),\sum_{i=k}^\infty Y_i^\epsilon(s)e_i)\rangle
        \widetilde{N}^{\epsilon^{-1}\varphi_\epsilon}(dv,ds)|\\
    &  &+
         \int_0^T\int_\mathbb{X}\sum_{i=k}^\infty|\langle \epsilon G(s,\widetilde{X}_s^\epsilon,v),e_i)\rangle|^2
         N^{\epsilon^{-1}\varphi_\epsilon}(dv,ds)\\
     &  &+
         2\sup_{t'\in[0,T]}|\int_0^{t'}\langle\sqrt{\epsilon}\sigma(s,\widetilde{X}_s^\epsilon)d\beta_s,\sum_{i=k}^\infty
         Y_i^\epsilon(s)e_i\rangle|\\
     &   &+
          2\int_0^T\|\sigma(s,\widetilde{X}_s^\epsilon)\psi_\epsilon(s)\|_H\|\sum_{i=k}^\infty Y_i^\epsilon(s)e_i\|_H
          ds\\
     &   &+
          \epsilon\int_0^T\|\sigma(s,\widetilde{X}_s^\epsilon)\|^2_{L_2(H)}ds\\
&=& I_1-I_2+2I_3+2I_4+I_5+2I_6+2I_7+I_8,
\end{eqnarray*}
by the Gronwall's inequality this implies that
\begin{eqnarray*}
    \sum_{i=k}^\infty|Y_i^\epsilon(t)|^2
\leq
    e^{-2\zeta_kt}[I_1+2I_3+2I_4+I_5+2I_6+2I_7+I_8].
\end{eqnarray*}

Hence, for any $t_0>0$, we have
\begin{eqnarray}\label{I1-5}
    \sup_{t_0\leq t\leq T}\sum_{i=k}^\infty|Y_i^\epsilon(t)|^2
\leq
    e^{-2\zeta_kt_0}[I_1+2I_3+2I_4+I_5+2I_6+2I_7+I_8].
\end{eqnarray}

\begin{eqnarray}\label{Eq.2-I3}
       \mathbb{E}I_3
 &\leq&
       \mathbb{E}\int_0^T\int_{\mathbb{X}}\|G(s,\widetilde{X}_s^\epsilon,v)\|_H\|\widetilde{X}_s^\epsilon\|_H|\varphi_\epsilon(s,v)-1|\nu(dv)ds\nonumber\\
 &\leq&
       \mathbb{E}\Big[\sup_{0\leq s\leq T}(\|\widetilde{X}_s^\epsilon\|_H+\|\widetilde{X}_s^\epsilon\|^2_H)
       \cdot
       \int_0^T\int_{\mathbb{X}}\|G(s,v)\|_{0,H}|\varphi_\epsilon(s,v)-1|\nu(dv)ds\Big]\\
 &\leq&
       \mathbb{E}\Big[(1+2\sup_{0\leq s\leq T}\|\widetilde{X}_s^\epsilon\|^2_H)\Big]
       \cdot
       \sup_{g\in S^N}\int_0^T\int_{\mathbb{X}}\|G(s,v)\|_{0,H}|g(s,v)-1|\nu(dv)ds\nonumber,
\end{eqnarray}

\begin{eqnarray}\label{Eq.2-I4}
     \mathbb{E}I_4
&\leq&
     \mathbb{E}\Big[\Big(\int_0^T\int_{\mathbb{X}}|\langle \epsilon G(s,\widetilde{X}_s^\epsilon,v),\sum_{i=k}^\infty
     Y_i^\epsilon(s)e_i)\rangle|^2N^{\epsilon^{-1}\varphi_\epsilon}(dv,ds)\Big)^{1/2}\Big]\nonumber\\
&\leq&
     \mathbb{E}\Big[\Big(\int_0^T\int_{\mathbb{X}}\epsilon^2\|
     G(s,\widetilde{X}_s^\epsilon,v)\|^2_H\|\widetilde{X}_s^\epsilon\|_H^2N^{\epsilon^{-1}\varphi_\epsilon}(dv,ds)\Big)^{1/2}\Big]\nonumber\\
&\leq&
     \mathbb{E}\Big[\Big(\epsilon^{1/2}\sup_{0\leq s\leq
     T}\|\widetilde{X}_s^\epsilon\|_H^2\Big)^{1/2}\Big(\int_0^T\int_{\mathbb{X}}\epsilon^{3/2}\|
     G(s,\widetilde{X}_s^\epsilon,v)\|^2_HN^{\epsilon^{-1}\varphi_\epsilon}(dv,ds)\Big)^{1/2}\Big]\nonumber\\
&\leq&
      \epsilon^{1/2}\mathbb{E}\Big(\sup_{0\leq s\leq T}\|\widetilde{X}_s^\epsilon\|_H^2\Big)
     +
      \epsilon^{1/2}\mathbb{E}\Big(\int_0^T\int_{\mathbb{X}}\|
      G(s,\widetilde{X}_s^\epsilon,v)\|^2_H\varphi_\epsilon(s,v)\nu(dv)ds\Big)\nonumber \\
&\leq&
       \epsilon^{1/2}\mathbb{E}\Big(\sup_{0\leq s\leq T}\|\widetilde{X}_s^\epsilon\|_H^2\Big)
     +
           \epsilon^{1/2}\mathbb{E}\Big[\Big(1+\sup_{0\leq s\leq T}\|\widetilde{X}_s^\epsilon\|_H\Big)^2\nonumber\\
      &&\quad \quad \quad \quad \quad \quad \quad \quad \quad \quad \quad \quad \quad \cdot
          \Big(\int_0^T\int_{\mathbb{X}}\|G(s,v)\|^2_{0,H}\varphi_\epsilon(s,v)\nu(dv)ds\Big)\Big]\nonumber\\
&\leq&
       \epsilon^{1/2}\mathbb{E}\Big(\sup_{0\leq s\leq T}\|\widetilde{X}_s^\epsilon\|_H^2\Big)
     +
           \epsilon^{1/2}\mathbb{E}\Big(1+\sup_{0\leq s\leq T}\|\widetilde{X}_s^\epsilon\|_H\Big)^2\nonumber\\
     &&\quad \quad \quad \quad \quad \quad \quad \quad \quad \quad \quad \quad \quad  \cdot
          \sup_{g\in S^N}\int_0^T\int_{\mathbb{X}}\|G(s,v)\|^2_{0,H}g(s,v)\nu(dv)ds,
     \end{eqnarray}

\begin{eqnarray}\label{Eq.2-I5}
   \mathbb{E}I_5
&\leq&
     \epsilon\mathbb{E}\Big(\int_0^T\int_{\mathbb{X}}\|
     G(s,\widetilde{X}_s^\epsilon,v)\|^2_H\varphi_\epsilon(s,v)\nu(dv)ds\Big)\nonumber\\
&\leq&
     \epsilon\mathbb{E}\Big(1+\sup_{0\leq s\leq T}\|\widetilde{X}_s^\epsilon\|_H\Big)^2
        \cdot
          \sup_{g\in S^N}\int_0^T\int_{\mathbb{X}}\|G(s,v)\|^2_{0,H}g(s,v)\nu(dv)ds,
\end{eqnarray}

\begin{eqnarray}\label{Eq.2-I6}
   \mathbb{E}I_6
&\leq&
   \sqrt{\epsilon}\mathbb{E}\Big\{[\int_0^T\|\sigma(s,\widetilde{X}_s^\epsilon)\|^2_{L_2(H)}\|\widetilde{X}_s^\epsilon\|^2_Hds]^{1/2}\Big\}\nonumber\\
&\leq&
   \sqrt{\epsilon}\mathbb{E}[\sup_{s\in[0,T]}\|\widetilde{X}_s^\epsilon\|^2_H+1]\times[\int_0^TK(s)ds]^{1/2},
\end{eqnarray}

\begin{eqnarray}\label{Eq.2-I7}
   \sup_{\psi_\epsilon\in \widetilde{S}^N}\mathbb{E}I_7
&\leq&
   \mathbb{E}\int_0^T\sqrt{K(s)}\sqrt{\|\widetilde{X}_s^\epsilon\|^2_H+1}\|\widetilde{X}_s^\epsilon\|_H\|\psi_\epsilon(s)\|_Hds\nonumber\\
&\leq&
     \mathbb{E}[\sup_{s\in[0,T]}\|\widetilde{X}_s^\epsilon\|^2_H+1]\times \Big(\int_0^TK(s)ds+N\Big),
\end{eqnarray}

\begin{eqnarray}\label{Eq.2-I8}
   \mathbb{E}I_8
&\leq&\epsilon
     \mathbb{E}[\sup_{s\in[0,T]}\|\widetilde{X}_s^\epsilon\|^2_H+1]\times \int_0^TK(s)ds,
\end{eqnarray}

%
%
By (\ref{I1-5})-(\ref{Eq.2-I8}), Lemma \ref{Lemma-Condition-0,H-1,H} and  (\ref{estimation-01}),
we have
$$
   \sup_{0<\epsilon\leq\epsilon_0}\sup_{\phi_\epsilon=(\psi_\epsilon,\varphi_\epsilon)\in\widetilde{\mathcal{U}}^N}\mathbb{E}\sup_{t_0\leq
   t\leq T}\sum_{i=k}^\infty|Y_i^\epsilon(t)|^2
\leq
    e^{-2\zeta_kt_0}C_N.
$$
\end{proof}

\begin{remark}\label{remark-04}
By Lemma \ref{Lemma-Condition-0,H-1,H}, Remark \ref{Remark-03},  (\ref{estimation-01}),
(\ref{estimation-02}) and the integrability of $K(s)$,
for every $0<\epsilon<\epsilon_0$ and $\eta>0$, there exists $t_0\in[0,T]$ such that
\begin{eqnarray}\label{estimation-02-01}
     (1/2-4\epsilon^{1/2})\mathbb{E}[\sup_{0\leq t\leq t_0}\sum_{i=k}^\infty|\langle
     \tilde{X}^{\epsilon}_{t},e_i\rangle|^2]
\leq
        \sum_{i=k}^\infty\langle \tilde{X}^{\epsilon}_{0},e_i\rangle^2
     +
            C\eta
     +
             C(\epsilon^{1/2}+\epsilon)
\end{eqnarray}
\end{remark}

In the sequel, the next two tightness results in $D([0,T],H)$ and $D([0,T],\mathbb{R})$ will be used.

\begin{lem}(\cite{Adam})\label{tight-H-cadlag}
Let $H$ be a separable Hilbert space with the inner product $\langle\cdot,\cdot \rangle$. For an orthonormal basis
$\{e_k\}_{k\in\mathbb{N}}$ in $H$, define the function $r^2_N:H\rightarrow R^+$ by
      $$
       r^2_N(x)=\sum_{k\geq N+1}\langle x,e_k\rangle^2,\ \ N\in\mathbb{N}.
      $$

Let $D$ be closed under addition which is a total subset of $H$.

Then the sequence $\{X^n\}_{n\in\mathbb{N}}$ of stochastic processes with trajectories in $D([0,T],H)$
is tight iff it is $D$-weakly tight and for every $\epsilon>0$
     \begin{eqnarray}\label{tight-H-cadlag-eq}
       \lim_{N\rightarrow\infty}\limsup_{n\rightarrow\infty}\mathbb{P}\Big(r^2_N(X^n(s))>\epsilon\ for\ some\ s,\
       0\leq s\leq T\Big)\rightarrow 0.
     \end{eqnarray}
\end{lem}

Here we say a $H$-valued sequence $\{X^n\}_{n\in\mathbb{N}}$ is`` $D$-weakly tight" (in (\cite{Adam})) if $\langle X^n, \phi \rangle$ as a $\mathbb{R}$-valued sequence is tight, for every $\phi\in D$.

In order to prove ``$D$-weakly tight" in Lemma \ref{tight-H-cadlag}, we need the tightness result in
$D([0,T],\mathbb{R})$;
and one can refer to \cite{Aldous}.

Let $\{Y^\epsilon\}_{\epsilon\in(0,\epsilon_0]}$ be a sequence of random elements of $D([0,T],\mathbb{R})$, and
$\{\tau_\epsilon,\delta_\epsilon\}$ be such that:

$(a)$ for each $\epsilon$, $\tau_\epsilon$ is a stopping time with respect to the
natural $\sigma$-fields, and takes only finitely many values.

$(b)$ for each $\epsilon$, the constant $\delta_\epsilon\in[0,T]$ and
$\delta_\epsilon\rightarrow 0$ as $\epsilon\rightarrow 0$.

We introduce the following condition on $\{Y^\epsilon\}$ : for each
sequence $\{\tau_\epsilon,\delta_\epsilon\}$ satisfying $(a)(b),$
\[\text{\textbf{Condition (A)}}\ \ \ \ \ \
\ Y^\epsilon(\tau_\epsilon+\delta_\epsilon)-Y^\epsilon(\tau_\epsilon)\rightarrow 0,\ as\ \epsilon\rightarrow0,\text{\
in\ probability}.
\]

For $f\in D([0,T],\mathbb{R})$, let $J(f)$ denote the maximum of the jump
$|f(t)-f(t-)|$.

\begin{lem}(\cite{Adam})\label{tight-R-cadlag}
Suppose that $\{Y^\epsilon\}_{\epsilon\in\mathbb{N}}$ satisfies \textbf{Condition $(A)$}, and either
$\{Y^\epsilon(0)\}$ and $\{J(Y^\epsilon)\}$ are tight on the line; or $\{Y^\epsilon(t)\}$
is tight on the line for each $t\in[0,T]$, then $\{Y^\epsilon\}$ is tight
in $D([0,T],\mathbb{R})$.
\end{lem}
\begin{thm}\label{Prop-2}
Fix $M\in\mathbb{N}$, and let $\phi_\epsilon=(\psi_\epsilon,\varphi_\epsilon),\phi=(\psi,\varphi)\in \widetilde{\mathcal{U}}^M$ be such that
 $\phi_\epsilon$ converges in distribution to $u$ as $\epsilon\rightarrow 0$. Then
 $$\mathcal{G}^{\epsilon}(\sqrt{\epsilon}\beta+\int_0^\cdot\psi_\epsilon(s)ds, \epsilon N^{\epsilon^{-1}\varphi_\epsilon})\Rightarrow \mathcal{G}^0(\int_0^\cdot\psi(s)ds, \nu^\varphi).$$

%
%
\end{thm}
\begin{proof}
 First, we prove that $\mathcal{G}^{\epsilon}(\sqrt{\epsilon}\beta+\int_0^\cdot\psi_\epsilon(s)ds, \epsilon N^{\epsilon^{-1}\varphi_\epsilon})$
is tight in $D([0,T],H)$. We will use Lemma \ref{tight-H-cadlag} and Lemma \ref{tight-R-cadlag}
to prove this result.

By (\ref{estimation-02-01}) and (\ref{estimation-03}), it follows that for any $\delta>0$ and $t_{0}\leq T$,
\begin{eqnarray*}
&&\sup_{0\leq \epsilon\leq\epsilon_{0}}\sup_{\phi_{\epsilon}=(\psi_{\epsilon},\varphi_{\epsilon})\in \tilde{\mathcal{U}}^{N}}
P(\sup_{0\leq t\leq T}\sum_{i=k}^{\infty}\langle\tilde{X}_{t}^{\epsilon},e_{i}\rangle^{2}>\delta)\\
&\leq& \delta^{-1}\sup_{0\leq \epsilon\leq\epsilon_{0}}\sup_{\phi_{\epsilon}=(\psi_{\epsilon},\varphi_{\epsilon})\in \tilde{\mathcal{U}}^{N}}
\mathbb{E}(\sup_{0\leq t\leq t_{0}}\sum_{i=k}^{\infty}\langle\tilde{X}_{t}^{\epsilon},e_{i}\rangle^{2})\\
&&+\delta^{-1}\sup_{0\leq \epsilon\leq\epsilon_{0}}\sup_{\phi_{\epsilon}=(\psi_{\epsilon},\varphi_{\epsilon})\in \tilde{\mathcal{U}}^{N}}
\mathbb{E}(\sup_{t_{0}\leq t\leq T}\sum_{i=k}^{\infty}\langle\tilde{X}_{t}^{\epsilon},e_{i}\rangle^{2})\\
&=:&(I)+(II).
\end{eqnarray*}

By estimates (\ref{estimation-02}), for any $\tilde{\epsilon}>0$, there exists $K_{1}>0$ and $\tilde{t}_{0}>0$, such that for any $k>K_{1}$ and $t_{0}<\tilde{t}_{0}$, we have $(I)\leq \frac{\tilde{\epsilon}}{2}$.

Fixing a constant $0<t_{0}\leq \tilde{t}_{0}$, by estimates (\ref{estimation-03}), we know that there exists constant $K_{2}>0$, such that for any $k>K_{2}$,
\begin{eqnarray}
(II)\leq  e^{-2\zeta_kt_0}C\leq \frac{\tilde{\epsilon}}{2}.
\end{eqnarray}

Since $\tilde{\epsilon}$, we have
$$
\lim_{k\rightarrow\infty}\sup_{0\leq \epsilon\leq\epsilon_{0}}\sup_{q_{\epsilon}=(\psi_{\epsilon},\varphi_{\epsilon})\in \tilde{\mathcal{U}}^{N}}
P(\sup_{0\leq t\leq T}\sum_{i=k}^{\infty}\langle\tilde{X}_{t}^{\epsilon},e_{i}\rangle^{2}>\delta)=0.
$$

 Hence by Lemma
\ref{tight-H-cadlag}, we only need to prove that, for every $\phi\in H_{\mathcal{A}^*}$,
$\langle\mathcal{G}^{\epsilon}(\sqrt{\epsilon}\beta+\int_0^\cdot\psi_\epsilon(s)ds, \epsilon N^{\epsilon^{-1}\varphi_\epsilon}),\phi\rangle$ is tight in
$D([0,T],\mathbb{R})$.

For convenience, denote $\langle\mathcal{G}^{\epsilon}(\sqrt{\epsilon}\beta+\int_0^\cdot\psi_\epsilon(s)ds, \epsilon N^{\epsilon^{-1}\varphi_\epsilon}),\phi\rangle$ by
$Y^\epsilon$.
We will check that $Y^\epsilon$ satisfies the condition of Lemma \ref{tight-R-cadlag}.

By (\ref{estimation-01}),
$$
\sup_{0<\epsilon<\epsilon_0}\mathbb{P}(|Y^\epsilon(t)|>L)\leq C/{L^2},
$$
hence $\{Y^\epsilon(t)\}$ is tight on the line for each $t\in[0,T]$.

Hence it remains to prove $Y^\epsilon$ satisfy \textbf{Condition $(A)$}. For each sequence
$\{\tau_\epsilon,\delta_\epsilon\}$ satisfying $(a)(b),$
   \begin{eqnarray}
       Y^\epsilon(\tau_\epsilon+\delta_\epsilon)-Y^\epsilon(\tau_\epsilon)
     &=&
         -\int_{\tau_\epsilon}^{\tau_\epsilon+\delta_\epsilon}\langle \widetilde{X}^\epsilon(s),\mathcal{A}^*\phi\rangle
         ds\nonumber\\
       &  &+
         \int_{\tau_\epsilon}^{\tau_\epsilon+\delta_\epsilon}\int_\mathbb{X}
                              \langle G(s,\widetilde{X}^\epsilon(s),v),\phi\rangle
              \Big(\epsilon N^{\epsilon^{-1}\varphi_\epsilon}(dv,ds)-\nu(dv)ds\Big)\nonumber\\
       &  &+
           \int_{\tau_\epsilon}^{\tau_\epsilon+\delta_\epsilon}\langle\sqrt{\epsilon}\sigma(s,\widetilde{X}^\epsilon(s))d\beta_s,\phi\rangle\nonumber\\
       &  &+
           \int_{\tau_\epsilon}^{\tau_\epsilon+\delta_\epsilon}\langle\sigma(s,\widetilde{X}^\epsilon(s))\psi_\epsilon(s),\phi\rangle ds\nonumber\\
     &=&
        -\int_{\tau_\epsilon}^{\tau_\epsilon+\delta_\epsilon}\langle \widetilde{X}^\epsilon(s),\mathcal{A}^*\phi\rangle
        ds\nonumber\\
       &  &+
         \int_{\tau_\epsilon}^{\tau_\epsilon+\delta_\epsilon}\int_\mathbb{X}
                              \langle G(s,\widetilde{X}^\epsilon(s),v),\phi\rangle
              \epsilon \widetilde{N}^{\epsilon^{-1}\varphi_\epsilon}(dv,ds)\nonumber\\
       &   &+
         \int_{\tau_\epsilon}^{\tau_\epsilon+\delta_\epsilon}\int_\mathbb{X}
                              \langle G(s,\widetilde{X}^\epsilon(s),v),\phi\rangle
              (\varphi_\epsilon-1)\nu(dv)ds \nonumber\\
       &  &+
           \int_{\tau_\epsilon}^{\tau_\epsilon+\delta_\epsilon}\langle\sqrt{\epsilon}\sigma(s,\widetilde{X}^\epsilon(s))d\beta_s,\phi\rangle\nonumber\\
       &  &+
           \int_{\tau_\epsilon}^{\tau_\epsilon+\delta_\epsilon}\langle\sigma(s,\widetilde{X}^\epsilon(s))\psi_\epsilon(s),\phi\rangle ds\nonumber\\
     &=&
        I_1+I_2+I_3+I_4+I_5,
   \end{eqnarray}

\begin{eqnarray}\label{Condition A-01}
       \mathbb{E}(|I_1|)
\leq
      \delta_\epsilon\|\mathcal{A}^*\phi\|^2_{H}+\mathbb{E}\int_{\tau_\epsilon}^{\tau_\epsilon+\delta_\epsilon}\|\widetilde{X}^\epsilon(s)\|^2_H
      ds
\leq
     \delta_\epsilon\Big(\|\mathcal{A}^*\phi\|^2_{H}+\mathbb{E}(\sup_{s\in[0,T]}\|\widetilde{X}^\epsilon(s)\|^2_H)\Big),
\end{eqnarray}

\begin{eqnarray}\label{Condition A-02}
       \mathbb{E}(|I_2|^2)
&\leq&
      \mathbb{E}\Big(\int_{\tau_\epsilon}^{\tau_\epsilon+\delta_\epsilon}\int_\mathbb{X}
                              \|\epsilon G(s,\widetilde{X}^\epsilon(s),v)\|_H^2\|\phi\|_H^2
              \epsilon^{-1}\varphi_\epsilon(v,s)\nu(dv)ds\Big)\\
&\leq&
      \epsilon\|\phi\|_H^2\Big(\mathbb{E}\sup_{s\in[0,T]}\|\widetilde{X}^\epsilon(s)\|^2_H+1\Big)\sup_{g\in
      S^M}\int_{0}^{T}\int_\mathbb{X}
                              \|G(s,v)\|_{0,H}^2
              g(v,s)\nu(dv)ds,\nonumber
\end{eqnarray}

\begin{eqnarray*}\label{Condition A-03-01}
       \mathbb{E}(|I_3|)
&\leq&
      \mathbb{E}\Big(\int_{\tau_\epsilon}^{\tau_\epsilon+\delta_\epsilon}\int_\mathbb{X}
                              \|G(s,\widetilde{X}^\epsilon(s),v)\|_H\|\phi\|_H
              |\varphi_\epsilon-1|\nu(dv)ds\Big)\nonumber\\
&\leq&
      \|\phi\|_H\mathbb{E}\Big((1+\sup_{s\in[0,T]}\|\widetilde{X}^\epsilon(s)\|_H)\int_{\tau_\epsilon}^{\tau_\epsilon+\delta_\epsilon}\int_\mathbb{X}
                              \|G(s,v)\|_{0,H}
              |\varphi_\epsilon-1|\nu(dv)ds\Big),
\end{eqnarray*}

By the same argument leading to (\ref{pf-main-thm-02-UI-02-02}), we can show that, for every $\eta>0$,
there exists $\delta>0$, such that
\begin{eqnarray*}\label{Condition A-03-02}
\sup_{g\in S^M}\int_A\int_\mathbb{X}\|G(s,v)\|_{0,H}|g(v,s)-1|\nu(dv)ds\leq \eta,\ \forall\ A\in[0,T],\
\lambda_T(A)\leq\delta.
\end{eqnarray*}

Hence, if $\delta_\epsilon\leq\delta$, we deduce that
\begin{eqnarray}\label{Condition A-03-03}
       \mathbb{E}(|I_3|)
&\leq&
      \eta\|\phi\|_H\mathbb{E}(1+\sup_{s\in[0,T]}\|\widetilde{X}^\epsilon(s)\|_H).
\end{eqnarray}

For the terms in $I_{4}$ and $I_{5}$, we have
\begin{eqnarray}\label{Condition A-04}
       \mathbb{E}(|I_4|^2)
&\leq&
      \epsilon\|\phi\|^2_H\mathbb{E}\int_0^T\|\sigma(s,\widetilde{X}^\epsilon(s))\|^2_{L_2(H)}ds\nonumber\\
&\leq&
      \epsilon\|\phi\|^2_H(\mathbb{E}[\sup_{s\in[0,T]}\|\widetilde{X}^\epsilon(s)\|^2_H]+1)\times\int_0^TK(s)ds;
\end{eqnarray}

\begin{eqnarray*}\label{Condition A-05-01}
       \mathbb{E}(|I_5|)
&\leq&
      \|\phi\|_H\mathbb{E}\int_{\tau_\epsilon}^{\tau_\epsilon+\delta_\epsilon}\|\sigma(s,\widetilde{X}^\epsilon(s))\|_{L_2(H)}\|\psi_\epsilon(s)\|_{H}ds\nonumber\\
&\leq&
      \|\phi\|_H\mathbb{E}\int_{\tau_\epsilon}^{\tau_\epsilon+\delta_\epsilon}\sqrt{K(s)}\sqrt{\|\widetilde{X}^\epsilon(s)\|^2_{H}+1}\|\psi_\epsilon(s)\|_{H}ds\nonumber\\
&\leq&
      M^{1/2}\|\phi\|_H\mathbb{E}\Big\{\sup_{s\in[0,T]}\sqrt{\|\widetilde{X}^\epsilon(s)\|^2_H+1}\times[\int_{\tau_\epsilon}^{\tau_\epsilon+\delta_\epsilon}K(s)ds]^{1/2}\Big\},
\end{eqnarray*}
where $M=\int_{0}^{T}\|\psi_{\epsilon}(s)\|_{H}^{2}ds$.

Since $K(\cdot)\in L^1([0,T],\mathbb{R})$, there exists $\widetilde{\delta}>0$ such that, if $\delta_\epsilon\leq\widetilde{\delta}$,
\begin{eqnarray}\label{Condition A-05-02}
       \mathbb{E}(|I_5|)
&\leq&
       \eta M^{1/2}\|\phi\|_H\mathbb{E}[\sqrt{\sup_{s\in[0,T]}\|\widetilde{X}^\epsilon(s)\|^2_H+1}].
\end{eqnarray}

By (\ref{Condition A-01})-(\ref{Condition A-05-02}) and Chebyshev inequality, we see that \textbf{Condition (A)} holds.
Thus we have proved that $\mathcal{G}^{\epsilon}(\sqrt{\epsilon}\beta+\int_0^\cdot\psi_\epsilon(s)ds, \epsilon N^{\epsilon^{-1}\varphi_\epsilon})$  is tight in $D([0,T],H)$.

\vspace{3mm}
Finally, we prove that $\mathcal{G}^0(\int_0^\cdot\psi(s)ds, \nu^\varphi)$ is the unique limit point of $\mathcal{G}^{\epsilon}(\sqrt{\epsilon}\beta+\int_0^\cdot\psi_\epsilon(s)ds, \epsilon N^{\epsilon^{-1}\varphi_\epsilon})$.

Note that $\widetilde{X}^\epsilon$ satisfies
\begin{eqnarray}
   \langle\widetilde{X}^\epsilon(t),\phi\rangle
&=&
      \langle X_0,\phi\rangle
      -
      \int_{0}^{t}\langle \widetilde{X}^\epsilon(s),\mathcal{A}^*\phi\rangle ds\nonumber\\
       &  &+
         \int_{0}^{t}\int_\mathbb{X}
                              \langle G(s,\widetilde{X}^\epsilon(s-),v),\phi\rangle
              \epsilon \widetilde{N}^{\epsilon^{-1}\varphi_\epsilon}(dv,ds)\nonumber\\
       &   &+
         \int_{0}^{t}\int_\mathbb{X}
                              \langle G(s,\widetilde{X}^\epsilon(s),v),\phi\rangle
              (\varphi_\epsilon-1)\nu(dv)ds\\
        &  &+
           \int_{0}^{t}\langle\sqrt{\epsilon}\sigma(s,\widetilde{X}^\epsilon(s))d\beta_s,\phi\rangle\nonumber\\
       &  &+
           \int_{0}^{t}\langle\sigma(s,\widetilde{X}^\epsilon(s))\psi_\epsilon(s),\phi\rangle ds\nonumber,\ \forall\phi\in V.
\end{eqnarray}

Denote $\overline{M}^\epsilon(t)=\int_{0}^{t}\sqrt{\epsilon}\sigma(s,\widetilde{X}^\epsilon(s))d\beta_s$ and $M^\epsilon(t)=\int_{0}^{t}\int_\mathbb{X}
                               G(s,\widetilde{X}^\epsilon(s-),v)
              \epsilon \widetilde{N}^{\epsilon^{-1}\varphi_\epsilon}(dv,ds)$.
Since
\begin{eqnarray*}
     \mathbb{E}(\sup_{s\in[0,T]}\|\overline{M}^\epsilon(s)\|^2_H)
&\leq&
     \epsilon\mathbb{E}\int_0^T\|\sigma(s,\widetilde{X}^\epsilon(s))\|^2_{L_2(H)}ds\\
&\leq&
     \epsilon\int_0^TK(s)ds(\mathbb{E}\sup_{s\in[0,T]}\|\widetilde{X}^\epsilon(s)\|^2_H+1),
\end{eqnarray*}
and
\begin{eqnarray*}
   \mathbb{E}(\sup_{s\in[0,T]}\|M^\epsilon(s)\|^2_H)
&\leq&
    \mathbb{E}\Big(\int_{0}^{T}\int_\mathbb{X}\|G(s,\widetilde{X}^\epsilon(s-),v)\epsilon\|^2_H\epsilon^{-1}\varphi_\epsilon\nu(dv)ds\Big)\\
&=&
    \epsilon\mathbb{E}\Big(\int_{0}^{T}\int_\mathbb{X}\frac{\|G(s,\widetilde{X}^\epsilon(s-),v)\|^2_H}{(1+\|\widetilde{X}^\epsilon(s-)\|_H)^2}(1+\|\widetilde{X}^\epsilon(s-)\|_H)^2\varphi_\epsilon\nu(dv)ds\Big)\\
&\leq&
     \epsilon\mathbb{E}\Big((1+\sup_{s\in[0,T]}\|\widetilde{X}^\epsilon(s)\|_H)^2\int_{0}^{T}\int_\mathbb{X}\|G(s,v)\|^2_{0,H}\varphi_\epsilon\nu(dv)ds\Big),
\end{eqnarray*}
by Lemma \ref{Lemma-Condition-0,H-1,H} and (\ref{estimation-01}),$\overline{M}^\epsilon\Rightarrow 0$ and $M^\epsilon\Rightarrow 0$, as
$\epsilon\rightarrow0.$

Choose a subsequence along which $(\widetilde{X}^\epsilon,u_\epsilon,\overline{M}^\epsilon,M^\epsilon)$ converges to $(\widetilde{X},u,0,0)$ in distribution. By the Skorokhod representation theorem, we may assume $(\widetilde{X}^\epsilon,u_\epsilon,\overline{M}^\epsilon,M^\epsilon)\rightarrow(\widetilde{X},u,0,0)$ almost surely

Note that convergence in Skorokhod topology to a continuous limit is equivalent to the uniform convergence,
and $C([0,T],H)$ is a closed subset of $D([0,T],H)$. Hence
$$
\lim_{\epsilon\rightarrow0}\sup_{s\in[0,T]}\|M^\epsilon(s)\|^2_H=0,\ \mathbb{P}-a.s..
$$

Since $\widetilde{X}^\epsilon-M^\epsilon\in C([0,T],H)$ and
$\widetilde{X}^\epsilon-M^\epsilon\rightarrow\widetilde{X}$ almost surely in $D([0,T],H)$, we have $\widetilde{X}\in
C([0,T],H)$, and
$$\lim_{\epsilon\rightarrow0}\sup_{s\in[0,T]}\|\widetilde{X}^\epsilon(s)-\widetilde{X}(s)\|^2_H=0,\
\mathbb{P}-a.s..$$

Along the same lines of the proof of Theorem \ref{Main-thm-01} and Proposition \ref{Prop-1}, letting, $\epsilon\rightarrow 0$, we see that $\widetilde{X}$
must
solve
\begin{eqnarray}
   \langle\widetilde{X}(t),\phi\rangle
&=&
      \langle X_0,\phi\rangle
      -
      \int_{0}^{t}\langle \widetilde{X}(s),\mathcal{A}^*\phi\rangle ds
      +
      \int_0^t\langle\sigma(s,\widetilde{X}(s))\psi(s),\phi\rangle ds\nonumber\\
       &   &+
         \int_{0}^{t}\int_\mathbb{X}
                              \langle G(s,\widetilde{X}(s),v),\phi\rangle
              (\varphi-1)\nu(dv)ds,\ \forall\phi\in V.
\end{eqnarray}

By the uniqueness, this gives that $\widetilde{X}=\mathcal{G}^0(\int_0^\cdot\psi(s)ds,\nu^\varphi)$. Proof is completed.
\end{proof}

%

\newpage

\def\refname{ References}

\end{document}